\documentclass[11pt,english]{article}
\usepackage{tikz}
\usepackage[T1]{fontenc}
\usepackage[latin9]{inputenc}
\usepackage[letterpaper]{geometry}
\usepackage{array}
\usepackage{amsthm}
\usepackage{amsmath}
\usepackage{setspace}
\usepackage{amssymb}
\usepackage{amsfonts}
\usepackage{graphicx}
\usepackage{cleveref}
\crefname{lemma}{Lemma}{lemmas}
\crefname{claim}{Claim}{claims}
\crefname{corollary}{Corollary}{corollaries}
\crefname{theorem}{Theorem}{theorems}
\crefname{definition}{Definition}{definitions}
\crefname{fact}{Fact}{facts}
\crefname{proposition}{Proposition}{propositions}
\crefname{conjecture}{Conjecture}{Conjectures}
\def\a{\alpha}
\def\g{\gamma}
\def\A{{\cal A}}
\def\B{{\cal B}}
\def\d{\delta}
\def\eps{\varepsilon}

\makeatletter

\newtheorem{conjecture}{Conjecture}[section]
\newtheorem{theorem}[conjecture]{Theorem}
\newtheorem{corollary}[conjecture]{Corollary}
\newtheorem{proposition}[conjecture]{Proposition}
\newtheorem{lemma}[conjecture]{Lemma}
\newtheorem{definition}[conjecture]{Definition}

\newtheorem{claim}[conjecture]{Claim}
\newtheorem{fact}[conjecture]{Fact}

\begin{document}

\title{Minimum degree thresholds for bipartite graph tiling}

\author{Albert Bush\\
                Yi Zhao\\
                Georgia State University\\}
\date{\today}
\maketitle

\begin{abstract}
For any bipartite graph $H$, we determine a minimum degree threshold for a balanced bipartite graph $G$ to contain a perfect $H$-tiling.  We show that this threshold is best possible up to a constant depending only on $H$.  Additionally, we prove a corresponding minimum degree threshold to guarantee that $G$ has an $H$-tiling missing only a constant number of vertices.  Our threshold for the perfect tiling depends on either the chromatic number $\chi(H)$ or the critical chromatic number $\chi_{cr}(H)$ while the threshold for the almost perfect tiling only depends on $\chi_{cr}(H)$.
Our results answer two questions of Zhao [{\it SIAM J. Disc. Math.} {\bf 23} (2009), 888-900]. They can be viewed as bipartite analogs to the results of Kuhn and Osthus [{\it Combinatorica} {\bf 29} (2009), 65-107] and of Shokoufandeh and Zhao [{\it Rand. Struc. Alg.} {\bf 23} (2003), 180-205].

%We answer a question of Zhao [{\it SIAM J. Disc. Math.} {\bf 23} (2009), 888-900] that determines the minimum degree threshold for a bipartite graph $G$ to contain an $H$-factor (a perfect tiling of $G$ with $H$) for any bipartite graph $H$.  We also show that this threshold is best possible up to a constant depending only on $H$.  This result can be viewed as an analog to Kuhn and Osthus' result [{\it Combinatorica} {\bf 29} (2009), 65-107] in that we show that either the chromatic number $\chi(H)$ or the critical chromatic number $\chi_{cr}(H)$ is the relevant parameter in bipartite graph tiling.  We also give a degree condition depending only on the critical chromatic number that guarantees an $H$-tiling that covers all but at most a constant number of vertices.
\end{abstract}

\section{Introduction}
%\textbf{History of Graph Tiling. }
Let $G$ be a graph on $n$ vertices and $H$ be a graph on $h$ vertices. The {\em tiling\/} (also called \emph{packing}) problem in extremal graph theory is to find in $G$ as many vertex-disjoint copies of $H$ as possible.  Researchers are interested in finding a tight minimum degree condition for $G$ to contain an $H$-factor -- a subgraph which consists of $\lfloor n/h\rfloor$ copies of $H$.  This is also sometimes called a perfect $H$-tiling or $H$-packing.  Dirac's theorem on Hamilton cycles~\cite{dirac} is one of the earliest tiling results.  It implies that every $n$-vertex graph $G$ with minimum degree $\delta(G)\ge n/2$ contains a perfect matching ($K_2$-factor).  The seminal result of Hajnal and Szemer\'edi \cite{hajnalszemeredi} determines the minimum degree threshold for a $K_r$-factor for all integers $r$.  By applying Szemer\'edi's Regularity Lemma \cite{regularitylemma}, Alon and Yuster \cite{alonyuster1, alonyuster2} found minimum degree conditions that guarantee an $H$-factor for an arbitrary $H$.  Koml\'os, Sark\"ozy, and Szemer\'edi \cite{komsarkszem} improved Alon-Yuster's result, giving a tight minimum degree for $H$ with equal-sized color classes.  Instead of using the chromatic number $\chi(H)$ as in \cite{alonyuster2, komsarkszem}, Koml\'os \cite{komlos} introduced the critical chromatic number $\chi_{cr}(H)$ and showed that it played a critical role in graph tiling (his result was improved by Shokoufandeh and Zhao \cite{zhao2}).  K\"uhn and Osthus \cite{kuhnosthus} finally determined exactly when the critical chromatic number or the chromatic number was the appropriate parameter. In order to accurately state their result, we need the following definitions.

For any graph $H$ on $h$ vertices, the critical chromatic number $\chi_{cr}(H)$ is defined as $\frac{(\chi(H) -1)h}{h-\sigma(H)}$, where $\sigma(H)$ is the size of the smallest color class over all proper $\chi(H)$-colorings of $H$.  Note that $\chi(H) -1 < \chi_{cr}(H) \le \chi(H)$ with equality if and only if all proper colorings of $H$ are balanced.  Suppose $H$ has connected components $C_1, \ldots , C_{k_c}$.  We define $hcf_c(H)$ as the highest common factor of integers $|C_1|, \ldots , |C_{k_c}|$.  Let $\ell = \chi(H)$.  Given a proper $\ell$-coloring $C$ of $H$ with $x_1 \le x_2 \le \ldots \le x_\ell$ as the sizes of the color classes, let $D(C)=\{x_{i+1}-x_i | i = 1, \ldots , \ell - 1\}$.  Let $D(H)=\cup D(C)$ where the union ranges over all proper $\ell$-colorings of $H$.  Now, $hcf_\chi(H)$ is the highest common factor of $D(H)$.  In particular, we set $hcf_\chi(H)=\infty$ if $D(H)=\{0\}$.  Lastly, we define the tiling indicator $hcf(H)$ as follows.  When $\chi(H) \ne 2$ and $hcf_{\chi}(H)=1$, we say $hcf(H)=1$.  If $\chi(H)=2$, we say $hcf(H)=1$ if and only if both $hcf_c(H)=1$ and $hcf_\chi(H) \le 2$.

\begin{theorem}[\cite{kuhnosthus}]
\label[theorem]{thm:ko}
For every graph $H$ on $h$ vertices, there exist integers $C$ and $m_0$ such that for all integers $m\ge m_0$, if $G$ is a graph on $n=mh$ vertices then the following holds.  If
\[ \delta(G) \ge \left\{
\begin{array}{rl}
( 1- 1/{\chi_{cr}(H)}) n + C & \text{if } hcf(H)=1\\
( 1- 1/{\chi(H)}) n + C & \text{otherwise},
\end{array} \right.
\]
then $G$ contains an $H$-factor.
\end{theorem}
It was also shown in \cite{kuhnosthus} that Theorem~\ref{thm:ko} is best possible up to the constant $C$. Other results and methods for tiling problems can be found in a recent survey of K\"{u}hn and Osthus \cite{KuOs-survey}.

\bigskip

Rather than working with an arbitrary graph $G$, one may restrict $G$ to be $r$-partite and tile it with some $r$-partite graph $H$.  Although it sounds like a special case, multipartite tiling is stronger than general tiling in the following sense.  First, a result on multipartite tiling does not follow from the corresponding general result. For example, an arbitrary graph $G$ of order $n$ contains a perfect matching if $\delta(G)\ge {n}/{2}$ (Dirac \cite{dirac}), while a bipartite graph $B$ with two partition sets of size $n/2$ contains a perfect matching if $\delta(B)\ge {n}/{4}$ (K\"{o}nig-Hall \cite{hall}).  Second, a result on multipartite tiling often implies one for general tiling. For example, suppose we know that every bipartite graph with two partition sets of size $n/2$ and minimum degree at least $n/4$ contains a perfect matching (assumed that $n$ is even). Let $G$ be an arbitrary graph $G$ with $\delta(G)\ge  n/2 + \epsilon n$ for some $\epsilon >0$. By taking a random, balanced, bipartition of $G$, we get a spanning bipartite subgraph $B$ with $\delta(B) \ge \frac{\delta(G)}{2} - o(n)\ge n/4$ (assuming $n$ is sufficiently large). Then $B$ contains a perfect matching, which is also a perfect matching of $G$.

In this paper we consider tiling in a balanced bipartite graph, where an $r$-partite graph is \emph{balanced} if all partition sets have the same size.
Zhao \cite{zhao1} determined the minimum degree threshold for a $K_{s,s}$-factor in a balanced bipartite graph for all $s$ (Hladk\'y and Schacht \cite{HlSc} and Czygrinow and DeBasio \cite{czy} later determined the minimum degree threshold for a $K_{s,t}$-factor). Given any bipartite $H$ of order $h$, since $K_{h,h}$ contains an $H$-factor, this gives a sufficient condition for an $H$-factor.
\begin{theorem}[\cite{zhao1}]
\label[theorem]{thm:zhao}
Let $H$ be a bipartite graph of order $h$. Suppose that $n$ is sufficiently large and divisible by $h$. If $G$ is a balanced bipartite graph on $2n$ vertices such that $\delta(G)\ge \frac{n}2 + \frac{3h}2 - 2$, then $G$ contains an $H$-factor.
\end{theorem}

We first show that Theorem~\ref{thm:zhao} is best possible (up to an additive constant) when $hcf(H) \ne 1$.

\begin{proposition}
\label[proposition]{thm:lb}
Let $H$ be a bipartite graph on $h$ vertices.  We assume $G$ to be a balanced bipartite graph on $2n=mh$ vertices where $m \in \mathbb{N}$.
\begin{enumerate}
    \item If $hcf(H) \ne 1$, then there exists a $G$ such that $\delta(G)=\lceil \frac{n}{2} \rceil - 1$ and $G$ does not contain an $H$-factor.
    \item If $hcf(H)=1$, then there exists a $G$ such that
\[ \delta(G)=\biggl( 1 - \dfrac{1}{\chi_{cr}(H)} \biggr)n-1 \]
and $G$ does not contain an $H$-factor.
\end{enumerate}
\end{proposition}

%So, throughout the paper, we will assume $hcf(H)=1$.  Thus, since our $H$ is always bipartite, $hcf(H)=1$ implies $hcf_c(H)=1$ and $hcf_\chi(H) \le 2$.

Zhao \cite{zhao1} asked about the minimum degree threshold for $H$-factors in bipartite graphs and suggested using either $\chi(H) (=2)$ or $\chi_{cr}(H)$, where the indicator function $hcf(H)$ determines which one is relevant. The main result of this paper answers this affirmatively; it can be viewed as a bipartite analog of \cref{thm:ko}.
%to combine the techniques in \cite{zhao1} while using the parameter $hcf(H)$ as in \cite{kuhnosthus}.
%
\begin{theorem}
\label[theorem]{thm:main}
Let $H$ be a bipartite graph on $h$ vertices such that $hcf(H)=1$.  If $G$ is a balanced bipartite graph on $2n=mh$ vertices, then there exist positive integers $m_0$ and $c_1(H)\le 4h^3$ such that if $m \ge m_0$ and
\[\delta(G) \ge \left(1-\dfrac{1}{\chi_{cr}(H)}\right)n+ c_1(H) \]
then $G$ contains an $H$-factor.
\end{theorem}

\cref{thm:lb}, Part 2, shows that \cref{thm:main} is best possible up to the value of $c_1(H)$.  Our constant $c_1(H)$ is on the order of $h^3$, and its exact value is specified in \eqref{eq:constant} of \cref{thm:extremalcase}.  Unlike the constant $C$ in ~\cref{thm:ko} which depends on the Regularity Lemma, our $c_1(H)$ is comparatively small.  Nevertheless, we are unable to determine the best possible value of $c_1(H)$ as in \cite{zhao1}.

Zhao \cite{zhao1} also asked for the minimum degree threshold for an almost perfect $H$-tiling. Koml\'os \cite{komlos} showed that for any graph $H$, every graph $G$ with $n$ vertices and $\d(G)\ge (1- 1/\chi_{cr}(H))n$ contains an $H$-tiling that covers all but at most $o(n)$ vertices. Shokoufandeh and Zhao \cite{zhao2} improved $o(n)$ to a constant, $O(h^2)$, where $h$ is the order of $H$. In this paper we prove a similar result for bipartite tiling.

%Additionally, by mimicking the proof of \cref{thm:main}, we are able to get an almost $H$-tiling for arbitrary bipartite $H$.  This is an analog of the result by  who showed that for any graph $H$, there exist $n_0$ and $C$ such that every graph $G$ with $n\ge n_0$ vertices and $\d(G)\ge (1- 1/\chi_{cr}(H))n$ contains an $H$-tiling that covers all but at most $C$ vertices.

\begin{theorem}\label[theorem]{thm:almost}
Let $H$ be a bipartite graph of order $h$. There exist integers $n_0$ and $c_2(H)< 8h^2$ such that every bipartite graph $G$ with $n \ge n_0$ vertices in each partition set contains an $H$-tiling that covers all but at most $c_2(H)$ vertices if $\d(G)\ge (1- 1/\chi_{cr}(H))n$.
\end{theorem}

It is important to note that K\"uhn and Osthus \cite{kuhnosthus} started their proof of \cref{thm:ko} with the result of Koml\'os (or the one of Shokoufandeh and Zhao), which gives an almost tiling of $G$, and then modified it into a perfect tiling under the assumption that $hcf(H)=1$. While proving \cref{thm:main}, we first find an almost-tiling (which leaves $o(n)$ vertices uncovered) from scratch. If $hcf(H)=1$, then we modify it into a perfect $H$-tiling, otherwise we modify it into an $H$-tiling that leaves only $O(h^2)$ vertices uncovered.

The structure of the paper is as follows. We prove \cref{thm:lb} in Section 2. In Section 3, we lay some groundwork for our proofs: we state bipartite versions of the Regularity Lemma and Blow-up Lemma.  Section 4 provides the proof of \cref{thm:main}, which is divided into the nonextremal case and the extremal case. Section 5 gives the proof of \cref{thm:almost} based on the one of \cref{thm:main}. In the last section we give concluding remarks, including a conjecture on $r$-partite tiling.

\medskip

\textbf{Notation.}  Fix a graph. For two vertices $x, y$, we write $x\sim y$ if $x$ is adjacent to $y$. Let $\Gamma(x)$ denote the set of neighbors of $x$ and $d(x)= |\Gamma(x)|$. For a vertex set $S$, let $\Gamma(x, S)= \Gamma(x)\cap S$. A bipartite graph $G[X,Y]$ means a bipartite graph with partition sets $X$ and $Y$. Given two disjoint subsets $A, B$ of $V(G)$, $G[A,B]$ is the bipartite subgraph induced on $A\cup B$ and its size is denoted by $e(A, B)$. The {\it density} of $A$ and $B$ is the {ratio} $d(A,B)=e(A,B)/(|A|\cdot |B|)$. When $A=\{x\}$, we simply write $d(x, B)$ instead of $d(\{x\}, B)$. Note that $d(\{x\}, B)$ is a density instead of a degree.

Throughout this paper we assume that $H$ is a bipartite graph on $h$ vertices such that $\sigma(H)=u$ and $h-\sigma(H)=w$. Let $C_1, \ldots, C_{k_c}$ be its connected components. Then each component $C_i$ has a unique 2-coloring $\{U_i, W_i\}$ with $|W_i|\ge |U_i|$. Let $c_i=|C_i|=|W_i|+|U_i|$ and $d_i = |W_i| - |U_i|$. Recall that $hcf_c(H)= hcf(c_1, \ldots, c_{k_c})$. We now define ${hcf_{\chi,c}(H)}$ as $hcf(d_1, \ldots , d_{k_c})$. When $hcf_c(H)=1$, there exist integers $\zeta_1, \ldots , \zeta_{k_c}$ such that $\sum{\zeta_i c_i}=1$. When $hcf_{\chi,c}(H)=1$, there exist integers $\beta_1, \ldots , \beta_{k_c}$ such that $\sum{\beta_i d_i}=1$.

The following elementary fact shows that we may choose coefficients $\zeta_i, \beta_i\le h$.  This will be used in Section~\ref{sec:ext} when we bound our constant $c_1(H)$. For completeness, we include its proof.
\begin{fact}\label[fact]{fact:gcd}
Let $k\ge 2$ and $a_1, \ldots , a_k$ be positive integers.  If $hcf(a_1, \ldots , a_k)=d$, then there exist integers $b_1, \ldots , b_k$ such that $\sum_{i=1}^k b_i a_i = d$ and $\max\{|b_1|, \dots, |b_k|\} \le \max\{a_1, \ldots a_k \}$.
\end{fact}

\begin{proof}
We prove by induction on $k$. Since $hcf(a_1,\ldots,a_k)=hcf(a_1, hcf(a_2,\ldots,a_k))$, it suffices to prove the case when $k=2$.  Let $a_1\le a_2$ be positive integers such that $hcf(a_1,a_2)=d$. Assume that $d<a_1< a_2$ otherwise $1\cdot a_1 + 0\cdot a_2=d$.
We will find positive integers $b_1$ and $b_2$ such that $a_1b_1 - d = a_2b_2$ with $\max\{b_1,b_2\} \le a_2$.  We let $a_1' = {a_1}/{d}$ and $a_2' = {a_2}/{d}$.  Then
$a'_1, a'_2 >1$ are integers with $hcf(a'_1, a'_2)=1$. Let $0< b'_1<a'_2$ be the multiplicative inverse of $a_1' \mod{a_2'}$ (note that $a'_2\ge 2$ implies that $b'_1\neq 0$). Then there exists an integer $b'_2$ such that
$a_1' b_1' -1  = b'_2 a_2'$. We derive that $0< b'_2< a'_1$ from $1\le b'_1<a'_2$ and $a'_1\ge 2$. It is easy to see that $b_1 = b'_1 d$ and $b_2 = b'_2 d$ are the desired integers.
\end{proof}
%
%
%4. The proof of Fact 6 can be simplified. We want to find positive
%integers x, y s.t. a x - d = b y, where a, b>0 and d= gcd(a,b).
%Let a'= a/d and b'=b/d. We have a' x' - 1 = b' y', where x' is the
%multiplicative inverse of a' mod b'. Then 0<x'< b' and consequently 0<y'<a'.
%Now x= x' d and y= y' d satisfy 0<x<b and 0<y<a.

%We first prove the assertion for $k=2$. Let $gcd(a_1,a_2)=d$ where $a_1 \le a_2$.  Assume $a_1 \neq a_2$, otherwise $1 \cdot a_1 + 0 \cdot a_2 = d$.  Thus, there exist integers $b_1$ and $b_2$ such that $b_1a_1+b_2a_2=d$.  Assume $b_1+b_2$ is minimal.  Since $a_1,a_2 > d$, we have that $b_1b_2 < 0$.  Assume $b_1 > a_2$.  Since $(b_1-a_2)a_1+(b_2+a_1)a_2=d$ then $b_2+a_1 < 0$.  This implies that $-b_2 > a_1$ .  In that case, we have $c_1a_1 + c_2a_2 = d$ and $c_1+c_2 < b_1+b_2$ where $c_1 = b_1 - a_2$ and $c_2=b_2+a_1$ which contradicts minimality.  If $b_2 > a_2$, then since $a_1 \le a_2$ and $b_1a_1+b_2a_2 = d > 0$, we must have $b_1 > b_2a_2/a_1 > a_2$, and we are in the previous case again.  Now, since $gcd(a_1,\ldots,a_k)=gcd(a_1,gcd(a_2,\ldots,a_k))$, the general assertion holds.
%\end{proof}

\begin{definition}\label[definition]{def:zetabeta}
Let $H$ be a bipartite graph with connected components $C_1, \ldots, C_{k_c}$. Suppose that $C_i= C_i[U_i, W_i]$ with $|W_i|\ge |U_i|$. Let $c_i=|W_i|+|U_i|$ and $d_i = |W_i| - |U_i|$.
\begin{enumerate}
    \item When $hcf_{c}(H) =1$, we define $\displaystyle\zeta(H)= \max_{1\le i\le k_c} |\zeta_i|$, where $\zeta_1, \dots, \zeta_{k_c}$ are integers such that $\sum_{1\le i\le k_c} \zeta_i c_i =1$ and $\max_{1\le i\le k_c} |\zeta_i| \le \max_{1\le i\le k_c} c_i\le h$.
%\begin{equation}\label{eq:zeta}
%\zeta(H) := \max_{1\le i\le k_c} |\zeta_i| \le \max_{1\le i\le k_c} c_i\le h.
%\end{equation}
    \item When $hcf_{\chi,c}(H) =1$, we define $\displaystyle\beta(H)= \max_{1\le i\le k_c} |\beta_i|$, where $\beta_1, \dots, \beta_{k_c}$ are integers such that $\sum_{1\le i\le k_c} {\beta_i d_i}=1$ and $\max_{1\le i\le k_c} |\beta_i| \le \max_{1\le i\le k_c} d_i \le w-u$.
%\begin{equation}\label{eq:beta}
%\beta := \max_{1\le i\le k_c} |\beta_i|\le \max_{1\le i\le k_c} d_i\le w-u.
%\end{equation}
\end{enumerate}
\end{definition}

\section{Proof of \cref{thm:lb}}

We first observe connections among $hcf_c(H)$, $hcf_{\chi}(H)$ and $hcf_{\chi, c}(H)$.

\begin{lemma}\label[lemma]{lem:hcfchic}
Let $H$ be any bipartite graph.
\begin{enumerate}
    \item Then $hcf_{\chi,c}(H) \le hcf_{\chi}(H) \le 2 \cdot hcf_{\chi,c}(H)$.
    \item If $hcf_{\chi,c}(H)=2$, then $hcf_c(H)\ge 2$.
    \item Suppose $hcf_c(H)=1$. Then $hcf_{\chi}(H)\le 2$ if and only if $hcf_{\chi,c}(H)=1$.
\end{enumerate}
\end{lemma}

\begin{proof} Suppose that $H$ has $k_c$ connected components $C_1[U_1, W_1], \dots, C_{k_c}[U_{k_c}, W_{k_c}]$. Let $c_i= |C_i|$ and $d_i= |W_i|- |U_i|$.

\textbf{Part 1.}
%Each connected component $C_i$ of $H$ has a unique $2$-coloring $\{U_i, W_i\}$.
We have $hcf_{\chi}(H) = hcf(A)$, where $A= \{ \sum_{i=1}^{k_c} e_i d_i : e_i \in \{-1, 1 \} \}$ is the set of all combinations of adding and subtracting $d_1, \ldots , d_{k_c}$. Therefore it suffices to show that
\[
hcf(d_1, \ldots , d_{k_c}) \le hcf(A) \le 2\cdot hcf(d_1, \ldots , d_{k_c}).
\]
In fact, letting $d= hcf(d_1, \ldots , d_{k_c})$ and $q= hcf(A)$, we have $d \le q$ because $d$ divides every element of $A$. On the other hand, for any $i$, $q$ divides $d_1 + \ldots + d_{k_c}$ and $d_1 + \dots + d_{i-1} - d_i + d_{i+1} + \dots  + d_{k_c}$ and thus $q$ divides $2d_i$. Therefore $q\le hcf(2d_1, \dots, 2d_{k_c})= 2d$.

\textbf{Part 2.} Suppose that $hcf_{\chi,c}(H)= 2$. Then for each component $C_i$ of $H$, $d_i$ is even.  This means $|U_i|$ and $|W_i|$ have the same parity and $c_i$ is even for all $i$. This implies that $hcf_c(H)\ge 2$.

\textbf{Part 3.} If $hcf_{\chi}(H)\le 2$, then by Part~1, $hcf_{\chi,c}(H)\le 2$. If $hcf_{\chi,c}(H)= 2$, then by Part~2, $hcf_c(H)\ge 2$ contradicting our assumption. Therefore $hcf_{\chi,c}(H)=1$. On the other hand, if $hcf_{\chi,c}(H)=1$, then $hcf_{\chi}(H)\le 2$ directly follows from Part~1.
\end{proof}

We now prove \cref{thm:lb} by using four constructions.
\begin{proof}[Proof of Proposition \ref{thm:lb}]
The proof consists of four (mutually disjoint) cases.  The first three cases together prove the existence of a graph $G$ with $\delta(G) = \lceil \frac{n}{2} \rceil - 1$ but containing no $H$-factor when $hcf(H) \neq 1$.  The last case provides a graph $G$ with $\delta(G) = [1-(1/\chi_{cr}(H))]n-1$ but containing no $H$-factor when $hcf(H)=1$.

\textbf{Case 1:} $hcf_c(H) \ge 3$.  Let $G=K_{\lceil \frac{n}{2} \rceil, \lfloor \frac{n}{2} \rfloor + 1} \cup K_{\lfloor \frac{n}{2} \rfloor, \lceil \frac{n}{2} \rceil - 1}$.  Since $hcf_c(H) \ge 3$, and any component of $H$ must fit entirely into one of the two connected components of $G$, we can deduce the following.  The size of the components of $G$ differ by $2$; but the size of the components of $H$ differ by multiples of $hcf_c(H)$ which is at least 3.  Thus, there is no way to arrange the components nor the copies of $H$ to even out the sizes of the components of $G$.  So $G$ contains no $H$-factor.

\textbf{Case 2:} $hcf_c(H)=2$.  Then each component of $H$ has an even size.
If $n$ is odd, let $G=K_{\lceil \frac{n}{2} \rceil, \lfloor \frac{n}{2} \rfloor} \cup K_{\lfloor \frac{n}{2} \rfloor, \lceil \frac{n}{2} \rceil}$.  If $n$ is even, let $G=K_{\frac{n}{2}, \frac{n}{2}+1} \cup K_{\frac{n}{2},\frac{n}{2} - 1}$.  In either case, since every component of $G$ has an odd size, $G$ does not contain an $H$-factor.

\textbf{Case 3:} $hcf_c(H)=1$ and $hcf_\chi(H) \ge 3$.  Let $G=K_{\lfloor \frac{n}{2} \rfloor + 1, \lceil \frac{n}{2} \rceil - 1} \cup K_{\lceil \frac{n}{2} \rceil - 1, \lfloor \frac{n}{2} \rfloor + 1}$.  It is an immediate consequence of \cref{lem:hcfchic} that if $hcf_\chi(H) \ge 3$ and $hcf_c(H)=1$, then $hcf_{\chi,c}(H) \ge 3$.  (Note that this does not imply $hcf_{\chi,c}(H) \ge hcf_\chi(H)$.)  Now, the sizes of the color classes of the connected components of $G$ differ by $1$ or $2$.  Since $hcf_{\chi,c}(H) \ge 3$, we can only adjust the relative sizes of the color classes of the connected components of $G$ by multiples of $hcf_{\chi,c}(H)$; so we can never get an $H$-factor.

\textbf{Case 4:} $hcf(H)=1$.  Recall that $|H|=h$, $u=\sigma(H)$, $w=h-\sigma(H)$, and $1-1/\chi_{cr}(H)=\frac{u}{h}$.  Let $G=K_{\frac{nu}{h}-1,\frac{nw}{h}+1} \cup K_{\frac{nw}{h}+1,\frac{nu}{h}-1}$.  Then $\delta(G) = [1-(1/\chi_{cr}(H))]n-1$.  Let $H$ be a graph with components $C_1,C_2, \ldots , C_{k_c}$.  By contradiction, suppose $G$ has an $H$-factor.  Then, one can see that
\[ \sigma(G) \ge m \sum_{i=1}^{k_c} \sigma(C_i) = mu. \]
This comes from the fact that one can simply arrange the $mk$ packed components of $G$ in the same way that one arranges the color classes of $G$ to attain $\sigma(G)$.  However, it is easy to see that $\sigma(G) = mu - 2$ by simply placing the $2$ components of size $\frac{nu}{h}-1 = \frac{mu}{2}-1$ in the same color class.  This is a contradiction.  So $G$ contains no $H$-factor.
\end{proof}

\section{Regularity Lemma and Other Tools}
The Regularity Lemma \cite{regularitylemma} and the Blow-up Lemma \cite{blowup} are the backbone of our proof.  They allow us to gain convenient structural properties from an arbitrary graph $G$.  Before stating the lemmas, we define $\epsilon$-regularity, and $(\epsilon, \delta)$-super-regularity.

\begin{definition}
Let $\epsilon, \delta > 0$.  Let $G$ be a graph with disjoint vertex sets $X$ and $Y$.  \textbf{(1)} We say the pair $(X,Y)$ is $\epsilon$-regular if for every $A \subseteq X$ and $B \subseteq Y$ satisfying $|A| > \epsilon |X|$, $|B| > \epsilon |Y|$ we have $|d(A,B)-d(X,Y)| < \epsilon$.  \textbf{(2)} The pair $(X,Y)$ is $(\epsilon, \delta)$-super-regular if $(X,Y)$ is $\epsilon$-regular and $d(x,Y) > \delta $ for every $x \in X$ and $d(y,X) > \delta $ for every $y \in Y$.
\end{definition}

The next two lemmas follow from the definition of $\epsilon$-regularity easily; their proofs can be found in the survey \cite{techreport}.

\begin{lemma}[Slicing Lemma]\label[lemma]{lem:slicinglemma}
Let $\epsilon , d > 0$ be constants.  Let $(X,Y)$ be an $\epsilon$-regular pair with density $d$.  For any $\gamma > \epsilon$, if $X' \subset X, Y' \subset Y$ and $|X'| \ge \gamma |X|, |Y'| \ge \gamma |Y|$, then $(X',Y')$ is an $\epsilon'$-regular pair with density $d'$ where $|d-d'| < \epsilon$ and $\epsilon'=\max\{ 2\epsilon, \frac{\epsilon}{\gamma} \}$.
\end{lemma}

\begin{lemma}[Embedding Lemma]\label[lemma]{lem:cliquesinregularpairs}
Let $1 > d \gg \epsilon > 0$.  If $(X,Y)$ is an $\epsilon$-regular pair with density $d$, then for any positive integers $a,b$, there exists an $n_0$ such that if $|X|,|Y| \ge n_0$, then $K_{a,b} \subset (X,Y)$.
\end{lemma}

%\begin{proof}
%We prove $(X',Y')$ is $\epsilon'$-regular.  Let $A,B$ be a subgraph of $G$ with $A \subset X'$ and $B \subset Y'$ with $|A| \ge \epsilon'|X'|$ and $|B| \ge \epsilon'|Y'|$, we prove $|d(A,B) - d'| < \epsilon'$.  A brief calculation shows that $|A| \ge \epsilon' |X'| \ge \epsilon' \gamma |X| \ge \epsilon |X|$.  Similarly, $|B| \ge \epsilon |Y|$.  Thus, $|d(A,B) - d| < \epsilon$.  So, $|d(A,B)-d'| < 2\epsilon \le \epsilon'$.
%\end{proof}

Now we are ready to state the bipartite form of Szemer\'edi's Regularity Lemma (see \cite{techreport} for a more detailed overview of the various forms of the Regularity Lemma).

\begin{lemma}[Regularity Lemma - Bipartite form]
\label[lemma]{lem:reg}
For every $\epsilon > 0$, there exists an $M \in \mathbb{R^+}$ such that if $G=(X,Y;E)$ is any bipartite graph with $|X|=|Y|=n$, and $d \in [0,1]$ is any real number, then there is a partition of $X$ into clusters $X_0, X_1, \ldots , X_k$, a partition of $Y$ into $Y_0, Y_1, \ldots , Y_k$, and a spanning subgraph $G'=(X,Y;E')$ with the following properties:

\begin{itemize}

\item $k \le M$
\item $|X_0|=|Y_0| \le \epsilon n$
\item $|X_i|=|Y_j|=N \le \epsilon n$ for all $1 \le i,j \le k$
\item $d_{G'}(v) > d_G(v)-(d+\epsilon)n$ for all $v \notin X_0 \cup Y_0$
\item All pairs $(X_i,Y_j)$, $1 \le i,j \le k$, are $\epsilon$-regular in $G'$, each with density either $0$ or greater than $d$.

\end{itemize}

\end{lemma}

The Blow-up Lemma is very useful for graph tiling, especially when combined with the Regularity Lemma as it essentially says that,  when embedding a graph of bounded maximum degree, an $(\epsilon, \delta)$-super-regular pair behaves like a complete bipartite graph. We only need the bipartite form of this lemma.

\begin{lemma}[Blow-up Lemma - Bipartite form]\label[lemma]{lem:blowup}
For every $\delta, \Delta > 0$, there exists an $\epsilon > 0$ such that the following holds.  Let $(X,Y)$ be an $(\epsilon, \delta)$-super-regular pair.  If a bipartite graph $H$ with $\Delta(H) \le \Delta$ can be embedded in $K_{|X|,|Y|}$, then $H$ can be embedded in $(X,Y)$.
\end{lemma}

%. Given a graph $G$ or order $n$ and parameters $\delta, \Delta > 0$, there exists an $\epsilon > 0$ such that the following holds:  Let $N$ be an arbitrary positive integer, and let us replace the vertices of $G$ with pairwise disjoint $N$-sets $V_1, V_2, \ldots , V_{n}$.  We construct two graphs on the same vertex set $V= \cup V_i$.  The graph $K(G)$ is obtained by replacing all edges of $G$ with copies of the complete bipartite graph $K_{N,N}$ and a less dense graph $G'$ is constructed by replacing the edges of $G$ with some $(\epsilon, \delta)$-super-regular pairs.  If a graph $H$ with maximum degree $\Delta(H) \le \Delta$ can be embedded into $K(G)$, then it can be embedded into $G'$.

\medskip

We now give a sufficient condition for a complete bipartite graph to contain an $H$-factor.

\begin{lemma}\label[lemma]{thm:newlemma9}
Let $H$ be a bipartite graph on $h$ vertices such that $hcf_{\chi, c}(H)=1$. Suppose that $\beta= \beta(H)$, $u=\sigma(H)$, and $w=h- u$.  Let $G=K_{mu+t,mw-t}$ such that $t=q(w-u)+r$ for nonnegative integers $q,m,t, r$ with $0\le r < w-u$. If $m \ge r\beta+q$ and $q \ge r\beta$, then $G$ contains an $H$-factor.
\end{lemma}

\begin{proof}
$K_{mu,mw}$ has a natural $H$-factor with all copies of $H$ having their smallest color classes on one side and the largest color classes on the other side.  We will show how to transform this into an $H$-factor of $G$.

First, since $m \ge q$ we can take $q$ copies of $H$ and swap their sides (here swapping means switching the sides of the color classes).  This now results in a spanning subgraph of $K_{mu+t-r,mw-t+r}$.  Let us call the part of this tiling that was not swapped as $G_1$ and the part that was swapped as $G_2$.  Since $hcf_{\chi,c}(H)=1$, there exist integers $\beta_1, \ldots , \beta_{k_c}$ as in \cref{def:zetabeta}.  Let us say that $\beta_1 , \ldots , \beta_i$ are nonnegative and $\beta_{i+1}, \ldots , \beta_{k_c}$ are all negative.  Now, in $G_1$ swap $r\beta_j$ copies of $C_j$ for all $j=1, \ldots , i$.  Note that since $m-q \ge r\beta$, we have enough copies of each component to perform the aforementioned swaps.    In $G_2$, swap $-r\beta_{j}$ copies of $C_{j}$ for all $j=i+1,\ldots , k_c$.  We can perform this swap because $q \ge r\beta$.  So, the left side gains
\[r = r\beta_1 d_1 + \ldots + r\beta_i d_i + r\beta_{i+1} d_{i+1} + \ldots + r\beta_{k_c} d_{k_c}
\]
vertices.  Similarly, the right side loses $r$ vertices, and we now have a spanning subgraph of $K_{mu+t,mw-t}=G$.
\end{proof}

We will use the following corollary of \cref{thm:newlemma9} in Section~\ref{sec:nec}, which is slightly stronger than the bipartite version of Lemma 12 in \cite{kuhnosthus}.

\begin{corollary}\label[corollary]{thm:kuhnosthuslemma}
Let $H$ be a bipartite graph on $h$ vertices such that $hcf(H)=1$.  Let $u=\sigma(H)$ and $w=h-\sigma(H)$.  Let $0 < \gamma < \frac{w-u}{u}$ be a constant.  Let $G[X,Y]$ be a complete bipartite graph on $mh$ vertices for some sufficiently large integer $m$ such that $(1+\gamma)\frac{u}{w} \le \frac{|X|}{|Y|} \le 1$.  Then $G$ contains an $H$-factor.
%\begin{itemize}
%\item $|G| \gg |H|$
%\item $G$ is a complete bipartite graph.
%\item $|X| + |Y|$ is divisible by $h$.
%
%\end{itemize}
\end{corollary}

\begin{proof}
We will prove that $G$ satisfies the conditions of \cref{thm:newlemma9} in order to get an $H$-factor.  First, since $|X|+|Y|$ is divisible by $h$, we may write $G=K_{mu+t,mw-t}$ where $m=(|X|+|Y|)/h$ and $t$ is some integer. Further write $t=q(w-u)+r$ for some integer $q$ and $0 \le r < w-u$.  Let
\begin{equation}\label{eq:meq}
m \ge \frac{(w-u)^2(h+u\gamma)\beta}{uw\gamma}.
\end{equation}
We must prove $m \ge r\beta + q$ and $q \ge r\beta$ with $\beta=\beta(H)$. Since $q=\lfloor \frac{t}{w-u} \rfloor \le \frac{t}{w-u}$, it is sufficient to prove that \textbf{(i)} $m \ge r\beta + \frac{t}{w-u}$ and \textbf{(ii)} $\frac{t}{w-u} \ge r\beta$.  
Since
%
%This is true for (i) because .  For (ii), note that since $r$ and $\beta$ are integers, if $\frac{t}{w-u} \ge r\beta$ then $\lfloor \frac{t}{w-u} \rfloor \ge r\beta$. Since
%
\[ \dfrac{|X|}{|Y|} = \dfrac{mu+t}{mw-t} \ge (1+\gamma)\dfrac{u}{w}, \]
we have that $th +tu \gamma \ge mwu\gamma$ which implies $t \ge \frac{uw}{h+u\gamma}m\gamma$.  Now, by \eqref{eq:meq}, we have
$\frac{uw}{h+u\gamma}m\gamma \ge (w-u)^2 \beta$, which implies that $t \ge (w-u)^2 \beta > (w-u)r\beta$ thus proving (ii).  On the other hand, $\frac{|X|}{|Y|} \le 1$ implies that $mu+t \le mw-t$, or $2t \le m(w-u)$.  Since $t \ge (w-u)r\beta$, we have $m(w-u) \ge (w-u)r\beta + t$, which gives (i).
\end{proof}
%
%For a fixed $\gamma$, we have $t \ge (w-u)^2 \beta > (w-u)r\beta$ if $\frac{uw}{h+u\gamma}m\gamma \ge (w-u)^2 \beta$.  This must be true since $|G|=mh \gg h$, thus proving (ii).   Since $\frac{|X|}{|Y|} \le 1$, we have $mu+t \le mw-t$.  So, $2t \le m(w-u)$.  Since by the above, $t \ge (w-u)r\beta$, we have
%
%\[ m(w-u) \ge (w-u)r\beta +t  \Leftrightarrow m \ge r\beta  + \dfrac{t}{w-u}. \]
%

\section{Proof of \cref{thm:main}}

Let $H$ be a bipartite graph on $h$ vertices with positive integers $u=\sigma(H)$ and $w=h-u$. We assume that $u<w$ otherwise $\chi_{cr}(H)=2$ and \cref{thm:zhao} gives the proof. We thus have $w \ge 2$, and $h \ge 3$.

The proof of our main theorem consists of two parts: the nonextremal case and the extremal case. Roughly speaking, a balanced bipartite graph with $2n=mh$ vertices is in the extremal case if it is relatively similar to $K_{\frac{nu}{h}-1,\frac{nw}{h}+1} \cup K_{\frac{nw}{h}+1,\frac{nu}{h}-1}$, the construction we gave in Case 4 of the proof of \cref{thm:lb}.
%This allows us to use the Regularity Lemma to tile $G$.  In the extremal case, we assume $G$ is relatively similar in structure to the aforementioned construction in Case 4.
\subsection{Nonextremal Case}
\label{sec:nec}

In this subsection we prove the following theorem, which covers the nonextremal case.
\begin{theorem}\label[theorem]{thm:nonextremalcase}
Let $H$ be a bipartite graph on $h$ vertices such that $hcf(H)=1$.  Let $u=\sigma(H)$ and $w=h-\sigma(H)$.  For every $\alpha >0$ there exist $\gamma>0$ and a positive integer $m_0$ such that if $m \ge m_0$ and $G[X,Y]$ is a balanced bipartite graph on $2n=mh$ vertices with
\[ \delta(G) \ge \left(1-\dfrac{1}{\chi_{cr}(H)} - \gamma \right)n \]
then $G$ either contains an $H$-factor or there exist sets $A\subset X$, $B \subset Y$ such that $|A|=|B|=\left\lfloor \dfrac{wn}{h} \right \rfloor$ and $d(A,B) \le \alpha$.
\end{theorem}
 We say that a bipartite graph $G[X,Y]$ is in the \textit{extremal case} with parameter $\alpha$ if there exist sets $A\subset X$, $B \subset Y$ such that $|A|=|B|=\lfloor \frac{wn}{h}\rfloor$ and $d(A,B) \le \alpha$.

%For ease of notation, we will always refer to $H$ as a grpah on $h$ vertices with $\sigma(H)=u$ and $h-\sigma(H)=w$.

%\subsection{Two Main Lemmas}

The proof of \cref{thm:nonextremalcase} is divided into two lemmas. The first lemma puts most vertices of $G$ into super-regular pairs such that the ratio of the sizes between the pairs is slightly larger than $u/w$. Having a ratio slightly larger than ${u}/{w}$ allows us to remove a small amount of vertices from the super-regular pair yet its remaining vertices can be tiled by $H$ perfectly by applying \cref{thm:kuhnosthuslemma} and Lemma~\ref{lem:blowup}. We make this precise by the following definition.

\begin{definition}\label[definition]{def:covering}
Given $0< \epsilon < d < 1$ and positive integers $p, q, N$, let $G[X,Y]$ be a balanced bipartite graph. A partition of $V(G)= X_0\cup Y_0\cup P_1\cup Q_1 \cup \dots P_k\cup Q_k$ is called an almost $(\epsilon, d, p, q, N)$-cover of $G$ if 

\begin{itemize}
\item $|X_0|,|Y_0| \le \epsilon n$
\item $X_0 \subset X$ and $Y_0 \subset Y$
\item For all $i$, $(P_i,Q_i)$ is $(\epsilon, d)$-super-regular
\item For all $i$, either $P_i \subset X$ and $Q_i \subset Y$, or $P_i \subset Y$ and $Q_i \subset X$
\item For all $i$, $|P_i|/p = |Q_i|/q \ge N$.
\end{itemize}
\end{definition}
%Let $0 < \epsilon \ll d \ll \gamma \ll 1$ be positive real numbers.  Let $k_0,N \gg w > u > 1$ be positive integers, and let $p:=2u/\gamma +w$, $q:=2w/\gamma$. Let $G[X,Y]$ be a bipartite graph on $2n$ vertices.  We say that $G$ has an \textit{almost $w,u$-covering} with parameters $\epsilon,d,\gamma,k_0,N$ if $G$ contains a spanning subgraph $G'$ consisting of clusters of vertices $X_0,Y_0,P_i,Q_i$ for $i=1, \ldots , k_0$ with the following properties:
%\begin{itemize}
%\item  and $|P_i|,|Q_i| \le N \le \epsilon n$ and $|P_i|,|Q_i| \gg u+w$ for $i=1,\ldots , k_0$
%\item  for $i=1,\ldots , k_0$ and $|P_i|/|Q_i|=p/q=(1+\frac{\gamma}{2})\frac{u}{w}$.
%\item For any cluster $C \in \{P_i,Q_i: i=1, \ldots , k_0 \}$, $N/q^2 \le |C| \le N$
%\end{itemize}

\begin{lemma}\label[lemma]{thm:broadlemma1}
Let $w>u$ be positive integers and $h=w+u$.  For every $\alpha > 0$ and integer $N$, there exists a positive integer $n_0$ and constants $0 < \epsilon \ll d \ll \gamma \ll \alpha$ such that if $G[X,Y]$ is a balanced bipartite graph on $2n$ vertices with $n \ge n_0$, and $\delta(G) \ge \left(\frac{u}{h} - \gamma \right)n$,
then either $G$ is in the extremal case with parameter $\alpha$ or
$G$ contains an almost $(\epsilon, d, p, q, N)$-cover, where $p=w + \frac{u}{\gamma}$ and $q=\frac{w}{\gamma}$ are integers.
\end{lemma}
%spanning subgraph $G'$ consisting of clusters of vertices $X_0,Y_0,P_i,Q_i$ $i=1,\ldots , k_0$ with the following properties:
%\begin{itemize}
%\item $|X_0|,|Y_0|,|P_i|,|Q_i| \le \epsilon n$ and $|P_i|,|Q_i| \gg h$ for $i=1,\ldots , k_0$ where $k_0 \le M_0$.
%\item $(P_i,Q_i)$ are $(\epsilon_1, d_1)$-super-regular for $i=1,\ldots , k_0$ and $|P_i|/|Q_i|=p/q=(1+\frac{\gamma}{2})\frac{u}{w}$ where $q:=2w/\gamma$, $p:=2u/\gamma+w$.
%\item For any cluster $C \in \{P_i,Q_i: i=1, \ldots , k_0 \}$, $N/q^2 \le |C| \le N$
%\end{itemize}

There are two reasons why we cannot immediately apply \cref{thm:kuhnosthuslemma} to each $(P_i, Q_i)$ in the cover.  First, we need to get rid of the \emph{exceptional sets} $X_0$ and $Y_0$.  Second, we may not have $|P_i|+|Q_i|$ divisible by $h$.  Achieving these two additional properties is the content of \cref{thm:broadlemma2}, in which we also assume $hcf(H)=1$. By the definition of $hcf(H)$ and part 3 of \cref{lem:hcfchic}, if $hcf(H)=1$ then $hcf_c(H)=1$ and $hcf_{\chi, c}(H)=1$. The condition of $hcf_c(H)=1$ is used for achieving the divisibility of $|P_i|+|Q_i|$. The condition of $hcf_{\chi, c}(H)=1$ is needed for \cref{thm:kuhnosthuslemma}.

\begin{lemma}\label[lemma]{thm:broadlemma2}
Let $H$ be a bipartite graph with $hcf_c(H)=1$ and $hcf_{\chi, c}(H)=1$. Let $u=\sigma(H)$ and $w=h-u$. Let $G$ be a balanced bipartite graph on $2n=mh$ vertices such that $\delta(G) \ge \left(1- 1/{\chi_{cr}(H)} - \gamma \right)n$. Suppose that $G$ contains an almost $(\epsilon, d, p, q, N)$-cover for some positive
$\epsilon \ll d \ll \gamma \ll 1$, integers $p, q$ satisfying $p/q=(1+\gamma){u}/{w}$, and sufficiently large $N$. Then $G$ contains an $H$-factor.
\end{lemma}

\begin{proof}[Proof of \cref{thm:broadlemma1}]
Note that we will omit the floor function when it does not affect our calculations.
Assume $n$ is large.  We may assume $\alpha \ll 1$.  We choose parameters $\epsilon_0$, $d_0$, $\gamma$ so that they satisfy the following relations
\begin{equation}\label{eq:elld}
    \epsilon_0 \ll d_0 \ll \gamma = \frac{1}{z} \ll \alpha 
\end{equation}
for some integer $z$. %$z \ge \frac{w}{2(w-u)}$.
Let $p=uz + w$ and $q=wz$ be two integers. Then $p$ and $q$ have the following property:
\begin{equation}\label{eq:pandq}
\dfrac{u}{w} < \dfrac{p}{q} = \dfrac{u}{w} + {\gamma} \le 1.
\end{equation}
%
%Lastly, we require that the parameters satisfy $\frac{128h^2q^4\epsilon}{u} < d_0 < \frac{\gamma}{4q^2}$. The significance of $q$ will be shown at the beginning of Step 4.

We apply the Regularity Lemma (\cref{lem:reg}) to with parameters $\epsilon_0$ and $d_0$ to $G$.  We obtain an integer $k_0\le M(\epsilon_0)$ and a spanning subgraph $G'$ consisting of clusters $X_1, Y_1, \ldots X_{k_0}, Y_{k_0}$ of size $N_0 \le \epsilon_0 n$ and exceptional sets $X_0$ and $Y_0$ of size at most $\epsilon_0 n$.  Every pair of clusters $(X_i, Y_j)$ is $\epsilon_0$-regular, with density either $0$ or greater than $d_0$.  The degrees of the vertices in $G'$ are very close to their degrees in $G$:
\[ d_{G'}(v) > d_{G}(v) - (d_0+\epsilon_0)n = \left(\dfrac{u}{h} - d_0 - \epsilon_0 -\gamma \right) n \]
%Our long-term goal will be to decompose $G'$ into super-regular cluster pairs whose sizes have ratio $\frac{p}{q}$. %
%

Let $R$ be the reduced graph of $G'$ where each vertex corresponds to a cluster in $G'-(X_0 \cup Y_0)$, and we say there is an edge between $X_i$ and $Y_j$ if the density $d(X_i, Y_j)>d_0$, written as $X_i\sim Y_j$.  Note that we use the same notation for a cluster in $G'$ and a vertex in $R$; we clearly say whether it is a cluster of $G'$ or a vertex of $R$ when this is not clear from the context. In order to bound $\delta(R)$, we consider an arbitrary $X_i$ and an arbitrary vertex  $x\in X_i$. We have
\begin{equation}\label{eq:dR}
    \left( \frac{u}{h}- \gamma - d_0 - 2\epsilon_0\right)n \le d_G(x) - (d_0+ \epsilon_0)n - \epsilon_0 n \le d_{G'}(x) - |Y_0|\le \sum_{Y_j \sim X_i} |Y_j| \le d_R(X_i) N_0.
\end{equation}
Using \eqref{eq:elld} and $k_0 N_0\le n$, we derive that $d_R(X_i)\ge (\frac{u}{h}- 2\gamma)k_0$. The same holds for any cluster in $Y$. Thus we have
\begin{equation} \label{eq:redmindeg}
\delta(R) \ge (\frac{u}{h}-2\gamma) k_0.
\end{equation}

We need a simple fact on the size of a maximum matching in bipartite graphs; for completeness, we include a proof.
\begin{fact}\label{fac:d}
If $G[X,Y]$ be a bipartite graph with minimum degree $\delta$ such that $|X| \le |Y|$, then $G$ has a matching of size at least $\min\{2\delta,|X|\}$.
\end{fact}

\begin{proof}
Let $M=\{x_1 y_1, \ldots , x_t y_t \}$ be a maximum matching in $G$.  Assume $t < |X|$.  Then, there exists a vertex $x \in X- \{x_1, \dots, x_t\}$.  Since $|Y| \ge |X|$, there also exists $y \in Y - \{y_1, \dots, y_t\}$.  Let $I= \{ 1\le i\le t: y_i \in \Gamma(x) \}$ and $J=\{1\le j\le t: x_j \in \Gamma(y)\}$. Then $|I|, |J|\ge \delta$. Since $M$ is a maximum matching, we have $I \cap J = \emptyset$ and $|I|,|J| \ge \delta$ (otherwise we may extend the matching).  This implies that $t \ge |I| + |J| \ge 2\delta$.
% and let $U_1 \subset X$ and $U_2 \subset Y$ be the unmatched vertices in $M$.  If $\delta < \frac{|X|}{2}$, consider the fact that in a maximum matching, unmatched vertices must have disjoint neighborhoods.  So, $|M| \ge |\Gamma(U_1)|+|\Gamma(U_2)| \ge 2\delta$ where $\Gamma(U_1)=\cup_{x\in U_1} \Gamma(x)$.  If $\delta \ge \frac{|X|}{2}$, we claim $|M| \ge |X|$.  Assume by contradiction $|M| < min\{2\delta,|X|\}$.  Then, there is a vertex $x \in X$ that is not included in a maximum matching.  Since $|Y| \ge |X|$, there must also be a vertex in $Y$ that is not included in the matching as well.  Their neighborhoods must be disjoint of size at least $\delta \ge \frac{X}{2}$, and they cannot be adjacent to any other unmatched vertices.  Thus, every vertex in $X$ is covered in the matching.
\end{proof}

Let $M$ be a maximum matching in the reduced graph $R$. Since $2u< w+u=h$, by Fact~\ref{fac:d}, we have $|M| \ge 2\delta(R) \ge 2(\frac{u}{h} -2\gamma ) k_0$. Denote by $U_1$ and $U_2$ the set of unmatched clusters from $X$ and $Y$ respectively.  Then $|U_1|,|U_2| \le (\frac{w-u}{h} + 2\gamma)k_0$.

%In Step 5, we will prove that that such integers will satisfy our decomposition lemma.
The next part of the proof will be decomposing clusters to get pairs of ratio $\frac{p}{q}$.  We first prove that we can find two disjoint subgraphs $P_1$ and $P_2$ of $R$ that satisfy the following properties.  The subgraph $P_1$ will have vertex sets $U_1$ and $\Gamma(U_1) := \cup_{X_i \in U_1} \Gamma(X_i)$.  Moreover, for any vertex $X_i \in U_1$, $d_{P_1}(X_i)=p$, and for any vertex $Y_j \in \Gamma(U_1)$, $d_{P_1}(Y_j) \le q-p$.  The subgraph $P_2$ will have vertex sets $U_2$ and $\Gamma(U_2)$; for any $Y_j \in U_2$, $d_{P_2}(Y_j)=p$, and for any $X_i \in \Gamma(U_2)$, $d_{P_2}(X_i) \le q-p$.  Note that since $M$ is maximal, $\Gamma(U_1), \Gamma(U_2) \subset V(M)$ and no edge of $M$ has one end in $\Gamma(U_1)$ and the other end in $\Gamma(U_2)$.  

%%%%%%%%%%%%%%%%%%%%%%%%%%%%%%%%%%%%%%%%%%%%%%%%%
%%%%%% Matching Picture 2 %%%%%%%%%%%%%%%%%%%%%%%
%%%%%%%%%%%%%%%%%%%%%%%%%%%%%%%%%%%%%%%%%%%%%%%%%

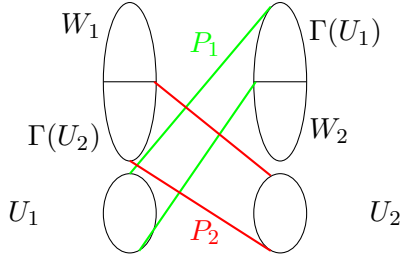
\begin{figure}[ht]
\begin{center}
\begin{tikzpicture}[scale=.5]

%\draw (0,2.75) circle (0pt) node{Finding $P_1$ and $P_2$};

\draw (-2.5,1.5) circle (0pt) node[left]{$W_1$};
\draw (2.5,-1.3) circle (0pt) node[right]{$W_2$};

\draw (-2,0) ellipse (20pt and 60pt) node[left=30pt]{};
\draw (2,0) ellipse (20pt and 60pt) node[right=30pt]{};

\draw (-2.7,0) -- (-1.3,0);
\draw (1.3,0) -- (2.7,0);

\draw (2.5,1.3) circle (0pt) node[right]{$\Gamma(U_1)$};
\draw (-2.5,-1.5) circle (0pt) node[left]{$\Gamma(U_2)$};

\draw (-2,-3.5) ellipse (20pt and 30pt) node[left=30pt]{$U_1$};
\draw (2,-3.5) ellipse (20pt and 30pt) node[right=30pt]{$U_2$};

\draw [green, thick] (-1.75,-4.5) -- (1.35,0);
\draw [green, thick] (-2,-2.45) -- (1.75,2);
\draw [green, thick] (0,1) circle (0pt) node{$P_1$};

\draw [red, thick] (1.75,-2.5) -- (-1.35,0);
\draw [red, thick] (1.75,-4.5) -- (-2,-2.1);
\draw [red,thick] (0,-4) circle (0pt) node{$P_2$};
\end{tikzpicture}
\caption{Finding $P_1$ and $P_2$}
\end{center}
\end{figure}
%
%We will find such a subgraph by the greedy algorithm.  Let $m$ be the number of vertices in $\Gamma(U_1)$ that have been chosen $q-p$ times.  It is enough to show that $\delta(R) \ge m+p$ because then each vertex in $U_1$ will have $p$ unique neighbors who are still available at each step in the process.  Since $m \le \frac{|U_1|p}{q-p}$, we must show that
%
Let $\alpha' = \alpha/12$.  We prove the following claim:
\begin{claim}\textbf{(a)} If $|U_1|,|U_2| \le (\frac{w-u}{h}-\alpha')k_0$, then the greedy algorithm suffices to find $P_1$ (or $P_2$).

\textbf{(b)} If $|U_1|,|U_2| > (\frac{w-u}{h}-\alpha')k_0$, then $G$ is in the extremal case with parameter $\alpha$.
\end{claim}

\begin{proof}
We first prove \textbf{(a)}.  We will only prove that we can find $P_1$ because the proof for $P_2$ is the same.  We will find $P_1$ by the greedy algorithm.  Arbitrarily order the vertices in $U_1$.  For each vertex in $U_1$, we find $p$ neighbors in $\Gamma(U_1)$ with the restriction that we cannot choose any vertex in $\Gamma(U_1)$ more than $q-p$ times.  When considering the $i$th vertex in $U_1$, suppose that there are $t$ vertices in $\Gamma(U_1)$ that have been chosen $q-p$ times.
%we must have $p$ vertices in $\Gamma(U_1)$ that have not been chosen $q-p$ times.
Since $t \le (i-1){p}/(q-p)< {|U_1|p}/(q-p)$, it suffices to show that $\delta(R) \ge p+ |U_1|p/(q-p)$.  Using \eqref{eq:redmindeg} and $|U_1|\le (\frac{w-u}{h}-\alpha')k_0$, we have
\[
\delta(R) - \dfrac{p}{q-p}|U_1| \ge \left( \dfrac{u}{h} - 2\gamma \right) k_0 - \frac{p}{q-p} \left ( \dfrac{w-u}{h} - \alpha' \right )k_0.
\]
From the Regularity Lemma, we have that $k_0 \ge \frac{1}{2 \epsilon_0}$.  Thus, it suffices to show that
\[
\phi:= \left( \dfrac{u}{h}- 2\gamma \right) - \left ( \dfrac{w-u}{h} - \alpha' \right) \dfrac{p}{q-p} \ge 2\epsilon_0 p. \]
In fact, the definition of $p, q$ and the assumption $z\ge \frac{2w}{w-u}$, which follows from $\g\ll 1$, give that
\[
\frac{p}{q-p} - \frac{u}{w-u} = \frac{uz+w}{(w-u)z-w} - \frac{u}{w-u} = \frac{w^2}{((w-u)z -w)(w-u)}\le \frac{2 w^2}{(w-u)^2 z}.
\]
By using \eqref{eq:elld}, we obtain that
\[
\phi \ge \left( \dfrac{u}{h}- 2\gamma \right) - \left ( \dfrac{w-u}{h} - \alpha' \right) \left( \frac{u}{w-u} + \frac{2 w^2}{(w-u)^2 z} \right)
> -2\gamma - \frac{2w^2\gamma}{h(w-u)} + \dfrac{u}{w-u}\alpha' \ge 2\epsilon_0 p.
\]
Thus, the greedy algorithm is sufficient to find the subgraphs $P_1$ and $P_2$.

\medskip

Now, we prove \textbf{(b)}.  We assume $|U_1|,|U_2| > (\frac{w-u}{h}-\alpha')k_0$.  Let $W_i$ be the neighbors of $\Gamma(U_i)$ in $M$ for $i=1,2$.  It is easy to see that the following four quantities must all be equal to $0$ or we can extend the matching in $G$:
\[ e(U_1, U_2)=e(U_1, W_2)=e(U_2,W_1)=e(W_1,W_2)=0. \]
For example, if there exists an edge $X_i Y_j$ between $W_1$ and $W_2$, then we can extend the matching as follows.  Let $Y_i$ denote the matched neighbor of $X_i$, $X_j$ denote the matched neighbor of $Y_j$, $X_{i'}$ denote a vertex in $U_1$ adjacent to $Y_i$, and $Y_{j'}$ denote a vertex in $U_2$ adjacent to $X_j$.  Then we can enlarge the matching by replacing $X_i Y_i, X_j Y_j$ by $X_{i'}Y_i$, $X_i Y_j$, and $X_j Y_{j'}$.

Now, letting ${\cal A}=U_1 \cup W_1$, and ${\cal B}=U_2 \cup W_2$, then $e_R(\A,\B)=0$.  Moreover,
\[
|\A| = |U_1| + |W_1| \ge |U_1| + \delta(R) \ge \left ( \dfrac{w-u}{h} - \alpha' \right )k_0 + \left ( \dfrac{u}{h}-2\gamma \right ) k_0 = \left ( \dfrac{w}{h}-\alpha' - 2\gamma \right ) k_0.
\]
Let $A'$ and $B'$ be the sets of vertices of $G$ in all the clusters of $\A$ and of $\B$ respectively. Since $k_0 N_0\ge (1-\epsilon_0) n$ and $\epsilon_0 \ll \gamma \ll \alpha'$,  we derive that $|A'|\ge (\frac{w}{h}-2\alpha')n$.  The same holds for $|B'|$. We also know that since $e_{G'}(A',B')=0$, 
\[
e_{G}(A', B')\le e_{G'}(A', B')+ |A'|(d_0 + \epsilon_0) n\le 2 d_0 n |A'| n.
\]
Now, by adding at most $2\alpha' n$ vertices to $A'$ and $B'$, we get two sets $A, B$ of size exactly $\lfloor \frac{w n}{h} \rfloor$; when $|A'|$ or $|B'|$ is greater than $\lfloor \frac{w n}{h} \rfloor$, we simply take a subset of size $\lfloor \frac{w n}{h} \rfloor$. %We claim that $d_G (A,B)\le \alpha'$.
Since each of the $2\alpha' n$ new vertices in $A$ (or $B$) might be adjacent to all the vertices in $B$ (or $A$),  we have
\[
d(A,B) \le \dfrac{e_{G}(A', B') + 2\alpha' n|B| + 2\alpha' n|A|}{|A||B|} = \dfrac{2d_0 n + 4\alpha' n}{|B|}  \le 12\alpha' = \alpha
\]
So, we are in the extremal case with parameter $\alpha$.
%Since we only deal with the nonextremal case in this section, we assume claim \textbf{(2)} is false and go under the assumption that we have found the subgraphs $P_1$ and $P_2$.
\end{proof}

We may now assume that we are not in the extremal case, and thus, Claim 4.6 \textbf{(a)} holds.  Now we use the structures of $P_1$ and $P_2$ to guide us to break up clusters.
In order to evenly divide a cluster into small pieces, we ensure the size of all clusters is divisible by $p q(q^2-p^2)$ by moving at most $p q(q^2-p^2)-1$ vertices from each cluster to the exceptional set. This increases $|X_0|$ and $|Y_0|$ by a constant, less than $p q(q^2-p^2) k_0$. For simplicity we still use $N_0$ for the size of the clusters.
%Thus $|X_0|, |Y_0| \le \epsilon_0 n + (pq(q^2-p^2)-1)k < 2 \epsilon_0 n$.

Now we only give the details on how to handle the clusters in $U_1\cup \Gamma(U_1)$.
We evenly decompose every cluster $X_i \in U_1$ into $p$ subclusters and adjoin each subcluster to a unique neighbor of $X_i$ in $P_1$.  Since $d_{P_1}(X_i)=p$ for each $X_i \in U_1$, this is possible.  However, we do not adjoin each subcluster of $X_i$ to the entire cluster.  Instead, we adjoin it to a subcluster of size $\frac{N_0}{q}$.  Thus, the ratio between two adjoining subclusters is $\frac{p}{q}$.

%%%%%%%%%%%%%%%%%%%%%%%%%%%%%%%%%%%%%%%%%%%%
%%%% Decomposing a Cluster %%%%%%%%%%%%%%%%%
%%%%%%%%%%%%%%%%%%%%%%%%%%%%%%%%%%%%%%%%%%%%

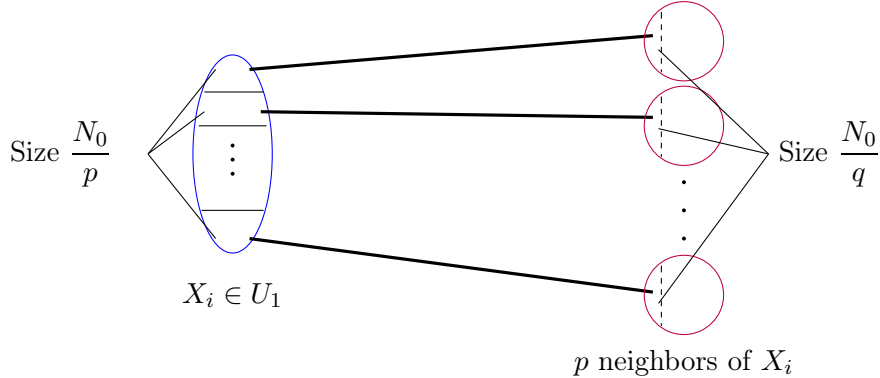
\begin{figure}[ht]
\begin{center}
\begin{tikzpicture}[scale=.75]

%\draw (-1.2,3.2) circle (0pt) node{Decomposing One Cluster in $U_1$};

\draw [blue] (-6, 0) ellipse (20pt and 50pt) node{};
\draw (-6.5,1.1) -- (-5.45,1.1);
\draw (-6.6,.5) -- (-5.4,.5);
\fill (-6, .15) circle (1pt) node{};
\fill (-6, -.1) circle (1pt) node{};
\fill (-6, -.35) circle (1pt) node{};
\draw (-6.55,-1) -- (-5.45,-1);

\draw (-6,-2.5) circle (0pt) node{$X_i\in U_1$};

%\draw [blue] (-1,2) circle (10pt) node{};
%\draw [blue] (-1,1) circle (10pt) node{};
%\fill (-1,0) circle (1pt) node{};
%\fill (-1,-.5) circle (1pt) node{};
%\fill (-1,-1) circle (1pt) node{};
%\draw [blue] (-1,-2) circle (10pt) node{};

\draw [style=very thick] (-5.7,1.5) -- (1.45,2.1);
\draw [style=very thick] (-5.5,.75) -- (1.45,.65);
\draw [style=very thick] (-5.7,-1.5) -- (1.45, -2.45);

\draw (-6.3,1.5) -- (-7.5,0);
\draw (-6.5,.75) -- (-7.5,0);
\draw (-6.3,-1.5) -- (-7.5,0);

\draw (-7.5,0) circle (0pt) node[left=10pt]{Size $\dfrac{N_0}{p}$};

%\draw (-2.5,0) node [anchor=mid west] {p $\left\{\rule{0cm}{2.25cm} \right.$ };

\draw [purple] (2,2) circle (20pt) node{};
\draw [purple] (2,.5) circle (20pt) node{};
\fill (2,-.5) circle (1pt) node{};
\fill (2,-1) circle (1pt) node{};
\fill (2,-1.5) circle (1pt) node{};
\draw [purple] (2,-2.5) circle (20pt) node{};

\draw (1.6,2.52) -- (1.6,2.4);
\draw (1.6,2.3) -- (1.6,2.2);
\draw (1.6,2) -- (1.6,2.1);
\draw (1.6,1.9) -- (1.6,1.8);
\draw (1.6,1.7) -- (1.6,1.6);
\draw (1.6,1.5) -- (1.6,1.45);

\draw (1.55,1.85) -- (3.5,0);

\draw (1.6,1.02) -- (1.6,.9);
\draw (1.6,.8) -- (1.6,.7);
\draw (1.6,.6) -- (1.6,.5);
\draw (1.6,.4) -- (1.6,.3);
\draw (1.6,.2) -- (1.6,.1);
\draw (1.6,0) -- (1.6,-.05);

\draw (1.55,.45) -- (3.5,0);

\draw (1.6,-1.98) -- (1.6,-2.1);
\draw (1.6,-2.2) -- (1.6,-2.3);
\draw (1.6,-2.4) -- (1.6,-2.5);
\draw (1.6,-2.6) -- (1.6,-2.7);
\draw (1.6,-2.8) -- (1.6,-2.9);
\draw (1.6,-3) -- (1.6,-3.05);

\draw (1.55, -2.65) -- (3.5,0);
\draw (3.5,0) circle (0pt) node[right] {Size $\dfrac{N_0}{q}$};

\draw (2,-3.7) circle (0pt) node {$p$ neighbors of $X_i$};

\end{tikzpicture}
\caption{Decomposing One Cluster in $U_1$}
\end{center}
\end{figure}
%
%Now, because $d(y) \le q-p$ for every $y \in \Gamma(U_1)$, no cluster gets chosen more than $q-p$ times.  Thus, at the end of the algorithm, the left over clusters in $\Gamma(U_1)$ have at least $N-(q-p)\frac{N}{q}=\frac{pN}{q}$ vertices.  We adjoin these clusters to their neighbors in the matching.

Let $Y_j\subset Y$ be a cluster covered by the matching $M$. We know that $Y_j$ has degree $i \le q-p$ in $P_1$ ($i=0$ when $Y_j\not\in \Gamma(U_1)$).  In total, $\frac{i N_0}{q}$ vertices of $Y_j$ are already used. %with $i$ subcluster of size $\frac{N_0}{p}$ from $U_1$ as mentioned above.
We match up the remaining $N_0-\frac{i N_0}{q}$ vertices in $Y_j$ with its neighbor $X_j$ in $M$ forming at most $3$ cluster pairs of ratio $\frac{p}{q}$ as follows. First take $\frac{i N_0}{q-p}$ vertices from $X_j$ and match them with $\frac{i p N_0}{q(q-p)}$ vertices from $Y_j$.  This makes a cluster pair with ratio $\frac{p}{q}$.  Now, the number of remaining vertices in $X_j$ is $N_0 - \frac{i N_0}{q-p}$, while the number of remaining vertices in $Y_j$ is $N_0 - \frac{i N_0}{q} - \frac{i p N_0}{q(q-p)}$, also equal to $N_0 - \frac{i N_0}{q-p}$. %Hence there are equal amount of vertices remaining from $X_j$ and $Y_j$.
Finally, we make two more cluster pairs with ratio $\frac{p}{q}$ by pairing together $(N_0-\frac{i N_0}{q-p})(\frac{p}{q+p})$ vertices from one cluster with $(N_0-\frac{i N_0}{q-p})(\frac{q}{q+p})$ from the other.

In summary, we broke all the clusters into subclusters and group them into pairs
with sizes
\begin{equation}\label{eq:sizes}
    \left\{\frac{N_0}{p},\frac{N_0}{q}\right\}, \left\{\frac{i N_0}{q-p},\frac{i p N_0}{q(q-p)}\right\}, \left\{\frac{q- p - i}{q-p} \frac{q}{p+q} N_0, \frac{q- p - i}{q-p}\frac{p}{p+q} N_0\right\},
\end{equation}
where $0 \le i \le q-p$. %Note that in the case when any of these sizes is not an integer, we use its floor function and move any leftover vertices in a cluster to $X_0$ or $Y_0$.

Let $\gamma'= \min \{ \frac{1}{q} , \frac{p}{q^2-p^2} \}$.
(then $\gamma' > d> \epsilon_0$ by \eqref{eq:elld}). The size of any subcluster is at least $\gamma' N_0$, which is larger than the given integer $N$ because $N_0\ge (1-2\epsilon_0)\frac{n}{k_0}$ is sufficiently large.  Let $(P_1, Q_1), \dots, (P_k, Q_k)$ denote these cluster pairs. After relabeling, we may assume that the first $k_1$ of them have $P_i$ in $X$ and $Q_i$ in $Y$ (see Figure~\ref{fig:postdec}). We have $k\le 2p k_0$ because each cluster in $U_1\cup U_2$ generates at most $p$ pairs, while each cluster covered by $M$ generate at most $3$ pairs, and $3\le p$.
The $\epsilon_0$-regularity between the original clusters implies that all $(P_i, Q_i)$ have density within $\epsilon_0$ of $d_0$. \cref{lem:slicinglemma} further guarantees that all $(P_i, Q_i)$ are $\epsilon_1$-regular with $\epsilon_1= \epsilon_0/ \gamma'$.

%To summarize, we have completely decomposed $G'$ into a new graph $R'$ such that $R'$ is a perfect matching of $k_0$ $\epsilon'$-regular pairs for some integer $k_0 \ge k$ where the relative size of each pair is exactly our desired ratio, $\frac{p}{q}$.

%%%%%%%%%%%%%%%%%%%%%%%%%%%%%%%%%%%%%%%%%%%%%%%%%%%
%%%% Post-Decomposition Picture %%%%%%%%%%%%%%%%%%%
%%%%%%%%%%%%%%%%%%%%%%%%%%%%%%%%%%%%%%%%%%%%%%%%%%%

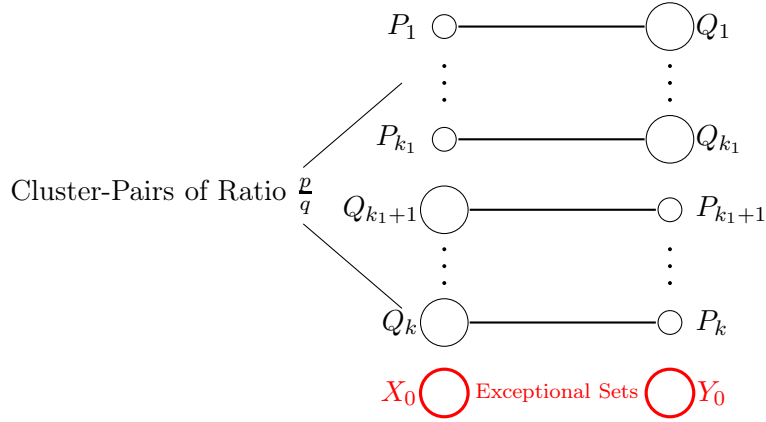
\begin{figure}[ht]
\begin{center}
\begin{tikzpicture}[scale=.75]

\draw (-2,3) circle (6pt) node[left=6pt]{$P_1$};
\draw (-2,1) circle (6pt) node[left=6pt]{$P_{k_1}$};
\draw (2,3) circle (12pt) node[right=6pt]{$Q_{1}$};
\draw (2,1) circle (12pt) node[right=6pt]{$Q_{k_1}$};

\draw (2,-.25) circle (6pt) node[right=6pt]{$P_{k_1+1}$};
\draw (2,-2.25) circle (6pt) node[right=6pt]{$P_{k}$};
\draw (-2,-.25) circle (12pt) node[left=6pt]{$Q_{k_1+1}$};
\draw (-2,-2.25) circle (12pt) node[left=6pt]{$Q_{k}$};

\draw (-7,0) circle (0pt) node{Cluster-Pairs of Ratio $\frac{p}{q}$};
\draw (-4.5,.5) -- (-2.75,2);
\draw (-4.5,-.5) -- (-2.75,-2);

\draw [red, style=very thick] (-2,-3.5) circle (12pt) node[left=6pt]{$X_0$};
\draw [red, style=very thick] (2,-3.5) circle (12pt) node[right=6pt]{$Y_0$};
\draw [red] (0,-3.5) circle(0pt) node{\scriptsize{Exceptional Sets}};

\draw [style=thick] (-1.75,3) -- (1.55,3);
\draw [style=thick] (-1.75,1) -- (1.55,1);
\draw [style=thick] (-1.55,-.25) -- (1.75,-.25);
\draw [style=thick] (-1.55,-2.25) -- (1.75,-2.25);

\fill (-2,2.3) circle (1pt) node{};
\fill (-2,2) circle (1pt) node{};
\fill (-2,1.7) circle (1pt) node{};
\fill (2,2.3) circle (1pt) node{};
\fill (2,2) circle (1pt) node{};
\fill (2,1.7) circle (1pt) node{};
\fill (2,-.95) circle (1pt) node{};
\fill (2,-1.25) circle (1pt) node{};
\fill (2,-1.55) circle (1pt) node{};
\fill (-2,-.95) circle (1pt) node{};
\fill (-2,-1.25) circle (1pt) node{};
\fill (-2,-1.55) circle (1pt) node{};

%\draw (0,4) circle (0pt) node{Graph $R'$ After Decomposition};

\end{tikzpicture}
\caption{Graph $G'$ After Decomposition}
\label{fig:postdec}
\end{center}
\end{figure}

\medskip

In order to obtain super-regularity for each $(P_i, Q_i)$, we now remove vertices with small degree into the opposite cluster to the exceptional sets $X_0, Y_0$.  Suppose that, for example, $P_i\subset X$ and $Q_i\subset Y$. We move any vertex $x \in P_i$ such that $d(x,Q_i) < d(P_i, Q_i) -\epsilon_1$ to $X_0$, and any vertex $y\in Q_i$ such that $d(y,P_i) < d(P_i, Q_i) -\epsilon_1$ to $Y_0$.  The $\epsilon_1$-regularity between $P_i$ and $Q_i$ guarantees that we move at most $\epsilon_1 |C|$ vertices from each $C\in \{P_i, Q_i\}$. In order to maintain the ratio to be exactly $\frac{p}{q}$, we may have to move more vertices from $P_i$ to $X_0$ and from $Q_i$ to $Y_i$ such that, in total, $P_i$ loses at most $p\lceil \epsilon_1 |P_i|/p \rceil\le \epsilon_1 |P_i| + p\le 2\epsilon_1 |P_i| $ vertices while $Q_i$ loses at most $q\lceil \epsilon_1 |Q_i|/q \rceil \le \epsilon_1 |Q_i| + q\le 2\epsilon_1 |Q_i|$ vertices.

We still denote the resulting clusters by $P_i$ and $Q_i$. Since the original $P_i$ has at least $\gamma' N_0$ vertices, the modified $P_i$ has at least $(1- 2\epsilon_1) \gamma' N_0$ vertices. By \cref{lem:slicinglemma}, the modified $(P_i, Q_i)$ is $2\epsilon_1$-regular. Since the density between the original $P_i$ and $Q_i$ is at least $d_0 - \epsilon_0$, the modified $(P_i, Q_i)$ satisfies
$d(x, Q_i)\ge d_0 - \epsilon_0 - 2\epsilon_1$ for any vertex $x\in P_i$, and
$d(y, P_i)\ge d_0 - \epsilon_0 - 2\epsilon_1$ for any vertex $y\in Q_i$. Let $\epsilon= 2\epsilon_1$ and $d= d_0 - \epsilon_0 - 2\epsilon_1$. Then all (current) $(P_i, Q_i)$ are $(\epsilon, d)$-super-regular.

In total, we moved at most $\sum_{C} (\epsilon_1 |C| + q) \le \epsilon_1 n+ kq$ vertices to $X_0$ where the sum ranges over all current clusters contained in $X$.  As a result, $|X_0|\le \epsilon_0 n + pq(q^2-p^2)k_0 + \epsilon_1 n + k q  \le \epsilon n$. The same holds for $|Y_0|$.
\end{proof}
%After removing at most $\epsilon' |P_i|$ vertices from $P_i$ and at most $\epsilon' |Q_i|$ vertices from $Q_i$, we may maintain that $|P_i|/|Q_i|=p/q$.
% and $b$ vertices from $Z_j$, we have two potential cases.  If $b > \frac{q}{p}a$, then remove $\frac{p}{q}b-a$ arbitrary vertices from $Z_i$ so that $Z_i$ loses a total of $\frac{p}{q}b$ vertices. Otherwise, $b \le \frac{q}{p}a$.  In this case, remove $\frac{q}{p}a - b$ arbitrary vertices from $Z_j$ so that $Z_j$ loses a total of $\frac{q}{p}a$ vertices.  Thus, we always maintain the ratio $\frac{p}{q}$.

%The smallest a cluster can be after the decomposition is either $\frac{N_0}{q}$ or $\frac{pN}{q^2-p^2}$ depending on the values of $p$ and $q$.  Both of these are bigger than $\frac{N_0}{q^2}$.  Thus, after ensuring super-regularity and ratio $\frac{p}{q}$ between each cluster, a cluster is always bigger than $\frac{N_0}{q^2} - \epsilon' N_0 \ge \frac{N_0}{2q^2} \gg h$.  Also, each vertex in a cluster still has a density into its adjoining cluster of at least $(d_0-2\epsilon')N_0 > \frac{d_0}{2}=d_0$.  Using the Slicing Lemma again, we see that each cluster pair is $(\epsilon_0,d_0)$-super-regular with $\epsilon_1 = 2\epsilon'$.

%
%Lastly, if $(Z_1,Z_2)$ and $(Z_1',Z_2')$ are two super-regular cluster pairs, then
%\[ \dfrac{|Z_i|}{|Z_i'|} \le \dfrac{N_0}{\frac{N_0}{2q^2}} = 2q^2. \]
%Let $\epsilon_1 = 2\epsilon'$, $\epsilon_2 = \epsilon''$, $\epsilon_3 = \frac{\epsilon}{2q^2}$, $D=d'$, and $\eta = \frac{\gamma}{2}$.  This concludes the proof of the lemma.

\medskip
%Now we prove \cref{thm:broadlemma2}, where the main task is to get rid of the exceptional sets and obtain divisibility of $|P_i|+ |Q_i$ by $h$ such that we can apply \cref{thm:kuhnosthuslemma} to each $(P_i, Q_i)$.

%%%%%%%%%%%%%%%%%%%%%%%%%%%%%%%%%%%%%%%%%%
%%%%%%%%%%%%%%%%%%%%%%%%%%%%%%%%%%%%%%%%%%
%%%%%%% Proof of Second Big Lemma %%%%%%%%
%%%%%%%%%%%%%%%%%%%%%%%%%%%%%%%%%%%%%%%%%%
%%%%%%%%%%%%%%%%%%%%%%%%%%%%%%%%%%%%%%%%%%

\begin{proof}[Proof of \cref{thm:broadlemma2}]
Let $X_0, Y_0, P_1, Q_1, \dots, P_k, Q_k$ be the given almost $(\epsilon, d, p, q, N)$-cover of $G$. As before, we call $X_0,Y_0$ exceptional sets, and $P_i, Q_i, i=1, \dots, k$ clusters. We know that $|X_0|, |Y_0|\le \epsilon n$, all pairs $(P_i, Q_i)$ are $(\epsilon, d)$-super-regular with $|P_i|/|Q_i|= \frac{p}{q} = \frac{u}{w}+ {\gamma} $.
Our first goal will be to take vertices in $X_0 \cup Y_0$ and find disjoint copies of $K_{u,w}$ (a supergraph of $H$) for each of them.

\begin{claim}\label{clm:X0Y0}
We may remove $|X_0 \cup Y_0|$ disjoint copies of $K_{u,w}$, each of which contains exactly one vertex from $X_0\cup Y_0$, such that each cluster $C\in \{P_i, Q_i\}$ loses at most $\frac{d}3|C|$ vertices.
\end{claim}

\begin{proof}
We say that a vertex $v$ is adjacent to a cluster $C$ (written as $v \sim C$) if $|\Gamma(v,C)| \ge d|C|$.  Following an arbitrary order of $X_0$ and $Y_0$, we associate each vertex $x \in X_0\cup Y_0$ to a cluster $C$ that $x$ is adjacent to. We also say that $x$ is associated with the cluster pair $(P_i, Q_i)$ if $C\in \{P_i, Q_i\}$. First assume that $C=P_i$. By \cref{lem:slicinglemma}, $(\Gamma(x, P_i), Q_i)$ is $\eps/d$-regular and by \cref{lem:cliquesinregularpairs}, $(\Gamma(x, P_i), Q_i)$ contains a copy of $K_{u, w-1}$ with $w-1$ vertices in $Q_i$. We then remove this copy of $K_{u,w-1}$ together with $x$ (they form a copy of $K_{u, w}$). When $C=Q_i$, we remove a copy of $K_{u-1, w}$ from $(P_i, \Gamma(x, Q_i))$ with $u-1$ vertices in $P_i$. Together with $x$, the removed vertices form a copy of $K_{u, w}$.

%We apply \cref{thm:broadlemma1} to get a subgraph $G'$ with the properties stated in the lemma.  The clusters in $G'$ are $(\epsilon,d)$-super-regular.  Their ratios are $\frac{p}{q}$ where $1 \ge \frac{p}{q} \ge \frac{u}{w}+ \frac{\gamma}{2}$. We move $x$ to $C$ and remove a copy of $K_{u,w}$, a supergraph of $H$, that is guaranteed by \cref{lem:cliquesinregularpairs}.  %ontaining $x$ as follows.  The cluster-pair $(Z_{i-k},Z_i)$ is $(\epsilon,d)$-super-regular.  So, since $x$ has $d|Z_i|$ neighbors in $Z_i$ and $d > \epsilon$, we have that $((\Gamma(x) \cap Z_i),Z_{i-k})$ is an $\frac{\epsilon}{d}$-regular pair with density $d'$ where $d-\epsilon < d' <d+\epsilon$ by \cref{lem:slicinglemma}.  Thus, it contains many copies of $K_{u,w}$.

To ensure that each cluster $C$ loses at most $\frac{d}3|C|$ vertices, we associate at most $\frac{d}{3w}|Q_i|$ vertices of $X_0\cup Y_0$ to any pair $(P_i, Q_i)$. Then $Q_i$ loses at most $\frac{d}{3}|Q_i|$ vertices because each associated vertex of $X_0\cup Y_0$ makes $Q_i$ lose at most $w$ vertices. On the other hand, $P_i$ loses at most $u$ vertices for each associated vertex. Since $|Q_i|/w \le |P_i|/u$, $P_i$ loses at most $u \frac{d}{3w}|Q_i| \le \frac{d}{3} |P_i|$ vertices.

We need to prove that under this restriction, there are enough clusters for all the vertices in the exceptional sets. First we give a lower bound for $\sum_{x\sim C}|C|$ for all $x\in X_0\cup Y_0$. Fix $x\in X_0$ (the case when $x\in Y_0$ is similar). By the minimum degree condition and the definition of $x\sim C$,
\[
(\frac{u}{h} - \g)n\le d_{G}(x)\le |Y_0| + \sum_{x\sim C}|C| + \sum_{C\subset Y: x\not\sim C} d|C| \le \epsilon n + d n + \sum_{x\sim C}|C|
\]
which implies that $\sum_{x\sim C}|C|\ge (\frac{u}{h} - 2\g)n$ by using $\epsilon \ll d\ll \g$. For a cluster $C\in \{P_i, Q_i\}$ with $x\sim C$, if we have associated $\frac{d}{3w}|Q_i|\ge \frac{d}{3w}|C|$ exceptional vertices with $(P_i,Q_i)$, then we can not associate $x$ with $C$.  If all the clusters $C$ adjacent to $x$ can not be used, then the number of exceptional vertices that have been considered is at least
\[
\sum_{x\sim C}\frac{d}{3w}|C|\ge \frac{d}{3w}(\frac{u}{h} - 2\g)n> 2\epsilon n,
\]
a contradiction.
\end{proof}

Other than a small number of copies of $K_{u,w}$, the graph $G$ now consists of cluster pairs $(P_i, Q_i)$ with ratio near $\frac{p}{q}$.  In order to apply \cref{thm:kuhnosthuslemma} to these $(P_i, Q_i)$, we want $|P_i|+|Q_i|$ to be divisible by $h$. We use the fact that $hcf_c(H)=1$ and let $\zeta= \zeta(H)$.

\begin{claim}\label{clm:div}
We may remove at most $2\zeta h k$ disjoint copies of $H$ such that each cluster $C\in \{P_i, Q_i\}$ loses at most $\zeta h^2$ vertices, and all $|P_i|+|Q_i|$ are divisible by $h$.
\end{claim}
\begin{proof}
Recall that $\sum_{1\le i\le k_c} \zeta_i c_i= 1$ and $\zeta= \max_{1\le i\le k_c} |\zeta_i|$, where $c_1, \dots, c_{k_c}$ are the sizes of the components of $H$.
After reordering, we may assume that $\zeta_1, \ldots , \zeta_j \ge 0$ and $\zeta_{j+1}, \ldots , \zeta_{k_c} < 0$
\begin{equation}
\zeta_1 c_1 + \ldots + \zeta_j c_j = 1 - \zeta_{j+1} c_{j+1} - \ldots - \zeta_{k_c} c_{k_c}
\label{gcdeqn}
\end{equation}
In order to ensure that the size of each cluster pair is divisible by $h$, we show how to increase or decrease the size of a cluster pair by 1 modulo $h$.  Let $G_1$ and $G_2$ denote the subgraphs induced by two cluster pairs $(P_{i},Q_{i})$ and $(P_{j}, Q_{j})$ respectively.  We will decrease the order of $G_1$ by $1$ modulo $h$ and increase the order of $G_2$ by $1$ modulo $h$.  To do this, we remove $2\zeta$ copies of $H$ by selectively choosing where the components of $H$ come from.  Since the cluster pairs are regular, we can find these copies of $H$ by \cref{lem:cliquesinregularpairs}.

From $G_1$ we remove $\zeta-\zeta_i$ copies of $C_i$ for $1\le i\le j$ and  $\zeta - \zeta_{i}$ copies of $C_{i}$ for $j< i\le k_c$.  By using \eqref{gcdeqn}, $G_1$ loses
\begin{align*}
& (\zeta-\zeta_1)c_1 + \ldots + (\zeta-\zeta_j)c_j + (\zeta - \zeta_{j+1})c_{j+1} \ldots + (\zeta - \zeta_{k_c})c_{k_c}
\\ & = \zeta(c_1 + \ldots + c_{k_c}) - (\zeta_1 c_1 + \ldots  + \zeta_{k_c} c_{k_c}) \\
& = \zeta(c_1 + \ldots + c_{k_c}) - 1 \ = \zeta \cdot h - 1
\end{align*}
vertices.  From $G_2$ we remove $\zeta+\zeta_i$ copies of $C_i$ for $1\le i\le j$ and  $\zeta+\zeta_{i}$ copies of $C_{i}$ for $j<i\le k_c$. A similar calculation shows that $G_2$ loses $\zeta \cdot h +1 \equiv 1$ (mod $h$). Since it is impossible that all the removed $\zeta h+1$ vertices come from one of $P_j$ and $Q_j$, each of $P_j, Q_j$ loses at most $\zeta h$ vertices.

Let $r_i$ be the remainder of $|P_i|+|Q_i| \mod{h}$ for $i=1, \dots, k$. Suppose that $r_i$ is the smallest nonzero remainder and $r_j$ is the largest remainder. By applying the procedure above at most $\min\{r_i, h-r_j\}$ times, we either reduce $r_i$ to $0$ or enlarge $r_j$ to $h$. Repeat this process at most $k-1$ times and obtain $r_i\equiv 0 \mod{h}$ for all $i=1, \dots, k$ (note that $\sum r_i\equiv 0 \mod{h}$ all the time). The total number of the removed copies of $H$ is at most $2\zeta (h-1)(k-1)< 2\zeta h k$, and each cluster loses at most $\zeta h (h-1)< \zeta h^2$ vertices.

Pairing $(P_i,Q_i)$ and $(P_j,Q_j)$ together and performing this process until either $r_i \equiv 0 \mod{h}$ or $r_j \equiv 0 \mod{h}$, it is easy to see that one may apply this procedure totally at most $(h-1)\sum_{i=1}^k r_i$ times to ensure that $|P_i|+|Q_i|$ is divisible by $h$ for all $i=1, \ldots , k$.  %times to all pairs while applying the procedure at most $h-1$ times to each pair, such that all $|P_i|+|Q_i|$ are divisible by $h$.

\end{proof}

Fix $i=1, \dots, k$. Let $P'_i, Q'_i$ denote the clusters obtained from $P_i, Q_i$ after applying Claim~\ref{clm:X0Y0} and Claim~\ref{clm:div}. We observe that $|P'_i|, |Q'_i|$ are large and $(1+\frac{\g}2)\frac{u}{w} \le \frac{|P_i|}{|Q_i|}\le 1$. In fact, by Claims~\ref{clm:X0Y0} and \ref{clm:div}, each cluster $C$ loses at most $d |C|/3 + \zeta h^2\le d |C|/2$ vertices, and consequently $|C'|\ge (1- \frac{d}2) |C|$. Since $d\ll \g\ll 1$, we derive that
\[
\left(1+\frac{\g}2 \right)\frac{u}{w}\le \left(1-\frac{d}2 \right)\left(\frac{u}{w}+ {\g}\right)= \frac{(1-\frac{d}2)|P_i|}{|Q_i|}\le \frac{|P'_i|}{|Q'_i|}\le \frac{|P_i|}{(1-\frac{d}2)|Q_i|}= \frac{\frac{u}{w}+ {\g}}{1-\frac{d}2}< 1
\]
By \cref{thm:kuhnosthuslemma}, the complete bipartite graph $K_{|P'_i|, |Q'_i|}$ contains an $H$-factor. If we can show that $(P'_i, Q'_i)$ is super-regular, then the Blow-up Lemma implies that $G[P'_i, Q'_i]$ also contains an $H$-factor. In fact, since $(P_i, Q_i)$ is $(\epsilon, d)$-super-regular, we have $|\Gamma(x, Q'_i)| \ge d |Q_i| - d|Q_i|/2\ge d|Q_i|/2$ for all $x\in P'_i$ and similarly $|\Gamma(y, P'_i)| \ge d|P'_i|/2$ for all $y\in Q'_i$. By the Slicing Lemma, $(P'_i, Q'_i)$ is $(2\epsilon, d/2)$-super-regular.

Note that $V(G) \setminus \bigcup_{i=1}^k (P_i\cup Q_i)$ consists of disjoint copies of $H$. We thus obtain the desired $H$-factor of $G$.
\end{proof}

\subsection{The Extremal Case}
\label{sec:ext}

We now prove that we can tile $G$ in the extremal case.  More precisely, we prove the following theorem:
\begin{theorem}\label[theorem]{thm:extremalcase}
Let $H$ be a bipartite graph with $hcf(H)=1$, $u=\sigma(H)$, $w=h-\sigma(H)$, $\zeta= \zeta(H)$, and $\beta= \beta(H)$. Let
\begin{equation}\label{eq:constant}
c_1(H) := \zeta h^2 + \beta(w-u)^2 + (\frac{h}{2}+1)(w-u)+w.
\end{equation}
Then, there exist $\alpha > 0$ and an integer $m_0$ such that for any $m \ge m_0$, if $G[X,Y]$ is a balanced, bipartite graph on $2n=mh$ vertices such that (i) $G$ has minimum degree
\[ \delta(G) \ge \left ( \dfrac{u}{h} \right ) n + c_1(H), \]
and (ii) there are subsets $A \subset X$, $B \subset Y$, where $|A|=|B|=\lfloor \frac{wn}{h} \rfloor$ with $d(A,B) \le \alpha$, then $G$ contains an $H$-factor.
\end{theorem}
By \cref{def:zetabeta}, we derive that $c_1(H)\le 4 h^3$ from \eqref{eq:constant},  and thus complete the proof \cref{thm:main}.

\medskip

To prove \cref{thm:extremalcase}, let us start with a simple corollary of the Blow-up Lemma.  We will use the notation $\delta(X,Y)$ to denote the minimum degree of a vertex in $X$ into a set $Y$.  In other words, $\delta(X,Y) = \min_{v\in X} |\Gamma(v,Y)|$.  Note that in general $\delta(X,Y) \neq \delta(Y,X)$. %and will be our primarily tool for tiling in the extremal case.
\begin{lemma}\label[lemma]{thm:modblowuplemma}
Let $\Delta$ be a positive integer.  There exists $0<\rho<1$ such that if a bipartite graph $F$ with $\Delta(F)\le \Delta$ can be embedded into $K_{|X|,|Y|}$, then it can be embedded into every bipartite graph $G[X,Y]$ with
\begin{equation}\label{eqn:1-r}
    \delta(X,Y) \ge (1-\rho)|Y|, \quad \delta(Y,X) \ge (1-\rho)|X|.
\end{equation}
\end{lemma}

\begin{proof}
We first prove that for any $0<\rho<1$, every bipartite graph $G[X,Y]$ satisfying (\ref{eqn:1-r}) is $\sqrt{\rho}$-regular.
In fact, consider subsets $A \subseteq X$, $B \subseteq Y$ with $|A| = \gamma_1 |X|$ and $|B| = \gamma_2|Y|$ for some $\gamma_1,\gamma_2 > \sqrt{\rho}$. By (\ref{eqn:1-r}), we have $\d(A, Y)\ge |Y|- \rho |Y|$ and consequently $\d(A, B)\ge |B|-\rho |Y|= (\gamma_2 - \rho) |Y|$. The density between $A$ and $B$ satisfies
\[ d(A,B)\ge \frac{\d(A,B) |A|}{|A| |B|}\ge \frac{ (\gamma_2 - \rho) |Y|}{|B|}= \frac{\gamma_2 - \rho}{\gamma_2}> 1- \frac{\rho}{\sqrt{\rho}}= 1- \sqrt{\rho}.
\]
Since $1-\sqrt{\rho}< d(A, B)\le 1$ and in particular, $1-\sqrt{\rho}< d(X, Y)\le 1$, we have
$|d(A, B)- d(X, Y)|< \sqrt{\rho}$.

Now assume that $K_{|X|,|Y|}$ contains a copy of $F$ and let $\epsilon$ be given by the Blow-up Lemma (Lemma~\ref{lem:blowup}) with $\delta=1/2$ and $\Delta(F)=\Delta$. Let $\rho=\min\{\epsilon^2, 1/2\}$ and $G[X,Y]$ be a bipartite graph satisfying (\ref{eqn:1-r}). Then $G$ is $(\epsilon, 1/2)$-super-regular and thus contains a copy of $F$.
\end{proof}

\medskip

\begin{proof}[Proof of \cref{thm:extremalcase}]
Recall that $A \subset X$ and $B \subset Y$ are sets of size $\lfloor \frac{wn}{h} \rfloor$ with $d(A,B) \le \alpha$. Let $A^c= X - A$ and $B^c = Y - B$. Then $|A^c|= |B^c|= \lceil \frac{u n}{h} \rceil$. %In this proof we will use $|A|$ and $|B|$, $|A^c|$ and $|B^c|$ interchangeably.

We define the following subsets:
\begin{align*}
A_1=\{x \in X : d(x,B) < \alpha^{\frac{1}{3}}|B| \}, & \quad B_1=\{y\in Y : d(y,A) < \alpha^{\frac{1}{3}}|A| \}\\
A_2=\{ x \in X : d(x,B) > (1-\alpha^{\frac{1}{3}})|B| \}, & \quad B_2=\{ y\in Y: d(y,A) > (1-\alpha^{\frac{1}{3}})|A| \}\\
A_0=X-A_1-A_2, & \quad B_0=Y-B_1-B_2.
\end{align*}
Clearly $A_1\cup A_2\cup A_0$ is a partition of $X$ and $B_1\cup B_2\cup B_0$ is a partition of $Y$. We claim that $A_1, B_1, A_2, B_2$ are very close to $A, B, A^c, B^c$ respectively (so $A_0$ and $B_0$ are fairly small) and subgraphs $G[A_1, B_2]$ and $G[A_2, B_1]$ are almost complete.

\begin{claim}\label[claim]{clm:size}
Assume that $\a^{\frac13}< \frac12$ and $\d(G)\ge \frac{u}{h} n$ (so $c_1(H)$ is unnecessary here).
\begin{enumerate}
\item $(1- \a^{\frac23}) |A|\le \{|A_1|, |B_1|\} \le (1+\a^{\frac23}) |A|$ and $|A^c| - \a^{\frac23} |A|\le \{|A_2|, |B_2|\} \le |A^c| + \a^{\frac23} |A|$.
%$|A- A_1|, |B-B_1|, |A^c- A_2|, |B^c - B_2|\le \a^{\frac23} |A|$.
%$(1-\alpha^{\frac{2}{3}})|A|\le |A_1|,|B_1| \ge (1+\alpha^{\frac{2}{3}})|A|$; $|A^c| - \a^{\frac23}|A|\le |A_2|,|B_2| \ge |A^c| + \alpha^{\frac{2}{3}}|A|$.

\item $\delta(B_2,A_1)\ge (1- 2\a^{\frac13})|A_1|$, $\delta(A_2,B_1)\ge (1-2\alpha^{\frac{1}{3}}) |B_1|$ and $\delta(A_1,B_2)\ge (1- 2\a^{\frac13} \frac{w}{u})|B_2|$, $\delta(B_1,A_2) \ge (1- 2\a^{\frac13} \frac{w}{u})|A_2|$.

%$\delta(B_2,A_1), \delta(A_2,B_1), \ge (1-\alpha^{\frac{1}{3}} - \alpha^{\frac{2}{3}}) |A|$ and $\delta(A_1,B_2), \delta(B_1,A_2) \ge |B^c| - (\alpha^{\frac{1}{3}}+ \alpha^{\frac23})|B|$.

\item $\Delta(B_1,A_1), \Delta(A_1,B_1) \le |A|(\alpha^{\frac{2}{3}}+\alpha^{\frac{1}{3}})$.

\item $|A_0|,|B_0| \le 2\alpha^{\frac{2}{3}} |A|$ and $\delta(A_0,B_1),\delta(B_0,A_1) \ge (\alpha^{\frac{1}{3}} - \alpha^{\frac{2}{3}})|A|$.
\end{enumerate}
\end{claim}

\begin{proof} Part~1. We only prove bounds for $|A_1|$ and $|A_2|$; the calculations for $|B_1|$ and $|B_2|$ are exactly the same. By definition of $A_1$,
\[ e(A-A_1,B) \ge \delta(A-A_1,B)|A-A_1| \ge \alpha^{\frac{1}{3}} |B| |A-A_1|. \]
On the other hand,
\[ e(A-A_1,B) \le e(A,B) \le \alpha |A| |B|. \]
Together they imply that
\[
|A-A_1| \alpha^{\frac{1}{3}} |B| \le \alpha |A||B| \ \Leftrightarrow \ |A-A_1| \le \alpha^{\frac{2}{3}}|A|.
\]
Since $|A|-|A_1| \le |A-A_1|$, we have $|A_1| \ge (1-\alpha^{\frac{2}{3}})|A|$.

In order to derive an upper bound for $|A^c-A_2|$, we need the minimum degree condition $\d(G)\ge \frac{u}{h} n$. Since $\d(G)$ is an integer, we actually have $\d(G)\ge \lceil \frac{u}{h} n\rceil$. Then
\[ e(B,A^c) = e(B,X)-e(A,B) \ge \lceil \tfrac{u}{h}n \rceil |B| - \alpha |A||B|. \]
Let $\bar{e}(B,A^c)$ denote the size of the bipartite complement of $G$ on $[B,A^c]$.  Since $\lfloor \tfrac{w}{h}n \rfloor + \lceil \tfrac{u}{h} n \rceil = n$, we have
\[ \bar{e}(B,A^c)= |B||A^c|-e(B,A^c) \le |B|(n-\lfloor \tfrac{w}{h}n \rfloor) - (\lceil \tfrac{u}{h}n \rceil |B| - \alpha |A||B|) = \alpha|A||B|. \]
By definition of $A_2$,
\[ e(A^c-A_2,B) \le (1-\alpha^{\frac{1}{3}})|B||A^c-A_2|. \]
Therefore,
\[ \bar{e}(A^c-A_2,B) \ge |A^c-A_2||B| - (1-\alpha^{\frac{1}{3}})|B||A^c-A_2| = \alpha^{\frac{1}{3}}|B||A^c-A_2|. \]
The upper and lower bounds for $\bar{e}(A^c-A_2,B)$ together imply that
\[ \alpha^{\frac{1}{3}}|B||A^c-A_2| \le \alpha |A||B| \Rightarrow |A^c-A_2| \le \alpha^{\frac{2}{3}}|A|. \]
We thus deduce that $|A_2| \ge |A^c| - \alpha^{\frac{2}{3}}|A|$.  Since $|A_0|+ |A_1| + |A_2| = n = |A| + |A^c|$, we further have $|A_0| + |A_1| \le |A| + \a^{\frac23}|A|$. Together with $|A_1| \ge (1-\alpha^{\frac{2}{3}})|A|$, it yields that $
|A| - \a^{\frac23}|A| \le |A_1|\le |A| + \a^{\frac23}|A|$.
The lower bound for $|A_1|$ also implies that $|A_2| \le |A^c| + \alpha^{\frac{2}{3}}|A|$. Together with $|A_2| \ge |A^c| - \alpha^{\frac{2}{3}}|A|$, we thus obtain desired bounds for $|A_2|$.

The proof above actually gives that
\[|A- A_1|, |B-B_1|, |A^c- A_2|, |B^c - B_2|\le \a^{\frac23} |A|.
\]

Part 2. Let us consider the minimum degree between $A_1$ and $B_2$ here; the same holds for the degree between $B_1$ and $A_2$. First $\delta(B_2,A_1) \ge \delta(B_2,A) - |A-A_1| \ge (1-\alpha^{\frac{1}{3}} - \alpha^{\frac{2}{3}}) |A|$. By using $\d(G)\ge \lceil \frac{u n}{h} \rceil = |B^c|$, we derive that
\[
\delta(A_1,B_2) \ge \delta(A_1,B^c) - |B^c-B_2| \ge \delta(G) - \alpha^{\frac{1}{3}}|B| - |B^c-B_2| \ge |B^c| - (\alpha^{\frac{1}{3}} + \alpha^{\frac{2}{3}}) |B|.
\]
We now prove that $\d(B_2, A_1)/ |A_1|\ge 1 - 2\a^{\frac13}$. By Part~1, $|A_1|\le (1+\a^{\frac23})|A|$. Then
\[\frac{\d(B_2, A_1)}{|A_1|}\ge \frac{(1-\a^{\frac13} - \a^{\frac23})|A|}{(1+\a^{\frac23})|A|}\ge 1- 2\a^{\frac13}
\]
because $\a^{\frac13}> 2\a^{\frac23}$.

Similarly we can prove $\d(A_1, B_2)/ |B_2|\ge 1 -2\a^{\frac13} \frac{w}{u}$ though we also need $|B|\le \frac{w}{h} n\le \frac{w}{u} |B^c|$:
\[\frac{\d(A_1, B_2)}{|B_2|}\ge \frac{|B^c|- (\a^{\frac13} + \a^{\frac23})|B|}{|B^c|+\a^{\frac23}|B|}\ge \frac{|B^c|- (\a^{\frac13} + \a^{\frac23})|B^c|\frac{w}{u}}{|B^c|+\a^{\frac23}|B^c|\frac{w}{u}}.
\]
By using $\a^{\frac13}> 2\a^{\frac23}$ again, we derive that $\d(A_1, B_2)/ |B_2|\ge 1- 2\a^{\frac13} \frac{w}{u}$.

\medskip

Part 3. By using $|A_1 - A|\le |A^c - A_2|\le \a^{\frac23} |A|$, we obtain
$\Delta(B_1,A_1)\le \Delta(B_1, A)+ |A_1 - A| \le (\alpha^{\frac{1}{3}} + \alpha^{\frac{2}{3}})|A|$. The same holds for $\Delta(A_1, B_1)$.

\medskip

Part 4. Part 1 immediately implies that $|A_0|,|B_0| \le 2\alpha^{\frac{2}{3}}|A|$. By definition of $A_1$, we have $\delta(A_0,B_1) \ge \a^{\frac13}|B| - |B- B_1|\ge (\alpha^{\frac{1}{3}} - \alpha^{\frac{2}{3}})|B|$. The same holds for $\delta(B_0, A_1)$.
\end{proof}

%%%%%%%%%%%%%%%%%%%%%%%%%%%%%%%%%%%%%%%%%%%%%%%%%%%%%%%%%
%%%%%%%%%%%%%%%%%%%%%%%%%%%%%%%%%%%%%%%%%%%%%%%%%%%%%%%%%
%%%%%%%%%%%%%%%%%%%%%%%%%%%%%%%%%%%%%%%%%%%%%%%%%%%%%%%%%
%%%%%%%%%%%%%%%%%%%%%%%%%%%%%%%%%%%%%%%%%%%%%%%%%%%%%%%%%
%%%%%%%%%% REAL WORK HERE %%%%%%%%%%%%%%%%%%%%%%%%%%%%%%%
%%%%%%%%%%%%%%%%%%%%%%%%%%%%%%%%%%%%%%%%%%%%%%%%%%%%%%%%%
%%%%%%%%%%%%%%%%%%%%%%%%%%%%%%%%%%%%%%%%%%%%%%%%%%%%%%%%%
%%%%%%%%%%%%%%%%%%%%%%%%%%%%%%%%%%%%%%%%%%%%%%%%%%%%%%%%%
%%%%%%%%%%%%%%%%%%%%%%%%%%%%%%%%%%%%%%%%%%%%%%%%%%%%%%%%%

Recall that $2n=mh$. %and thus $h$ always divides $2n$.
We now separate the proof into two parts, when $m$ is even and when $m$ is odd. We give all details in Part~1, including the exact values of $\a$ and $n$, and while reducing Part~2 to Part~1, we only justify the value of $c_1(H)$.
%each of which will have several cases.

\medskip

\noindent\textbf{Part I: $m$ is even.}
Apply \cref{thm:modblowuplemma} with $F := H$ to obtain a constant $0 < \rho < 1$.  We define $\a>0$ such that
\begin{equation}\label{eqn:alpha}
\a^{\frac13}= \min\left\{\dfrac{1}{5h^2}, \frac{\rho}{2h} \right\},
\end{equation}
With $\zeta= \zeta(H)$, since we chose $m_0$ sufficiently large, we may assume $m \ge {2\zeta h^2}/{\alpha^{\frac{2}{3}}}$ so that $n= m h /2$ satisfies
\begin{equation}\label{eqn:n}
n\a^{\frac23}\ge \zeta h^3.
\end{equation}
Let $G_1 = G[A_1, B_2\cup B_0]$ and $G_2= G[B_1, A_2\cup A_0]$ denote the induced subgraphs of $G$ on $A_1\cup B_2\cup B_0$ and $B_1\cup A_2\cup A_0$, respectively.
Our first step is to remove some copies of $H$ so that the orders of $G_1$ and $G_2$ are divisible by $h$.

Suppose that $v(G_1) \equiv r \,\,(mod\,\,\, h)$ and accordingly $v(G_2) \equiv -r \,\,(mod\,\,\, h)$ for some $0 \le r < h$.
\begin{claim}\label[claim]{thm:hcfclaim}
We may remove $2r\zeta$ copies of $H$ from $G$ where $r\zeta h+r$ vertices come from $G_1$ and $r\zeta h - r$ vertices come from $G_2$.  On the other hand, $r\zeta h$ vertices are from each of $X$ and $Y$.
%we can make $v(G_1)$ and $v(G_2)$ both divisible by $h$.
\end{claim}

\begin{proof}
We first note that since $G_1[A_1, B_2]$ and $G_2[A_2,B_1]$ are almost complete, we may find many disjoint copies of $H$ from them. In fact, since $|A_1|/|B_2|$ is about $w/u$, $K_{|A_1|, |B_2|}$ contains an $H$-tiling that covers most of its vertices. By \cref{clm:size},  $\delta(B_2,A_1) \ge (1- 2\alpha^{\frac{1}{3}})|A_1|$ and $\delta(A_1,B_2)\ge (1 - 2\alpha^{\frac{1}{3}}\frac{w}{u})|B_2|$. By (\ref{eqn:alpha}), $2\a^{\frac13}\frac{w}{u}\le \rho$. \cref{thm:modblowuplemma} thus implies that $G_1[A_1, B_2]$ contains an $H$-tiling that covers most of its vertices.

%let $A'_1, B'_2$ be the subsets of $A_1, B_2$ by omitting any $r\zeta h\le \a^{\frac23} |B|$ vertices from each set. In order to find a $K_{u,w}$ from $G[A'_1, B'_2]$, we pick any $u$ vertices $v_1, \dots, v_u\in A'_1$. By the bound for $\d(A_1, B_2)$ given in \cref{clm:size}, any $u$ vertices of $A'_1$ share at least $(|1- u 2\a^{\frac13})|B_2|$ neighbors in $B_2$. Since $|B'_2|= |B_2|- r\zeta h$, by (\ref{eqn:alpha}) and (\ref{eqn:n}), we have $(|1- u 2\a^{\frac13})|B_2|\ge w + r\zeta h $. Since $\d(A'_1, B'_2)\ge \d(A'_1, B_2)- r\zeta h\ge w$, we can find $w$ vertices from $B'_2$ that are adjacent to all $v_1, \dots, v_u$ and thus complete the copy of $K_{u,w}$.

We remove $2r\zeta$ copies of $H$ as follows:  from $G_1[A_1, B_2]$, remove $r(\zeta+\zeta_i)$ copies of $C_i$, and from $G_2[A_2,B_1]$, remove $r(\zeta-\zeta_i)$ copies of $C_i$ for all $i=1, \ldots , k_c$.  Now fix an index $i$. Note that $r(\zeta+\zeta_i)$ and $r(\zeta-\zeta_i)$ have the same parity.  If they are even, then we remove $r(\zeta+\zeta_i)/2$ copies of $C_i$ from $G_1$ with the larger side in $X$, and the other $r(\zeta+\zeta_i)/2$ copies of $C_i$ from $G_2$ with the smaller side in $X$.  Similarly, remove $r(\zeta-\zeta_i)/2$ copies of $C_i$ from $G_2$ with the larger side in $X$, and the other copies of $H$ with the smaller side in $X$.  Clearly $X$ and $Y$ lose the same number of vertices for each $i$. Since at the end $X$ and $Y$ together lose $2r\zeta h$ vertices, each of them loses $r\zeta h$ vertices. If $r(\zeta+\zeta_i)$ is odd, then remove $\lceil r(\zeta+\zeta_i)/2 \rceil$ copies of $C_i$ from $G_1$ with the larger side in $X$ and $\lfloor r(\zeta+\zeta_i)/2 \rfloor$ copies of $C_i$ from $G_1$ with the smaller side in $X$ (therefore $X$ loses $w_i - u_i$ more vertices than $Y$).  On the other hand, we remove $\lfloor r(\zeta-\zeta_i)/2 \rfloor$ copies of $C_i$ from $G_2$ with the larger side in $X$ and $\lceil r(\zeta-\zeta_i)/2 \rceil$ copies of $C_i$ from $G_2$ with the smaller side in $X$ (this makes $Y$ lose $w_i - u_i$ more vertices than $X$). Thus $X$ and $Y$ again lose the same number of vertices: each loses $r\zeta h$ vertices at the end.  The total number of vertices that $G_1$ loses is
\[
r(\zeta+\zeta_1)c_1 + \ldots + r(\zeta+\zeta_{k_c})c_{k_c} = r\zeta(c_1 + \ldots + c_{k_c}) + r(\zeta_1c_1 + \ldots + \zeta_{k_c}c_{k_c})  = r\zeta h + r
\]
%$G_1$ loses $r$ vertices modulo $h$.  Hence, both graphs are divisible by $h$.
A similar calculation shows that $G_2$ loses $r\zeta h -r$ vertices.
\end{proof}

Denote the sets of the remaining vertices in $X, Y, A_1, A_2, B_1, B_2$ by $X', Y', A'_1, A'_2, B'_1, B'_2$, respectively. The difference between $|A_1|$ and $|A'_1|$ (similarly between $|B_2|$ and $|B'_2|$, etc.) is at most $r\zeta h$. Our choice (\ref{eqn:n}) of $n$ is equivalent to $w h^2 \zeta \le \a^{\frac23} \frac{w}{h} n$. Since $r\le h-1$ and $|A|= \frac{w}{h}n$, we derive that
\begin{equation}\label{eqn:rzh}
w r\zeta h \le \a^{\frac23} |A|.
\end{equation}

Let $\tilde{A_2}= A'_2 \cup A_0$ and $\tilde{B_2}= B'_2 \cup B_0$.
The current $G_1, G_2$ are $G_1[A'_1, \tilde{B_2}]$ and $G_2[B'_1, \tilde{A_2}]$, respectively. By \cref{thm:hcfclaim},  both $v(G_1)$ and $v(G_2)$ are divisible by $h$. Let $m_1= v(G_1)/h$, $m_2= v(G_2)/h$, and write
\[
|A'_1|=m_1w + s, \quad |B'_1|=m_2w+t, \quad |\tilde{A_2}|=m_2u-t, \quad |\tilde{B_2}|=m_1u-s
\]
for some integers $s$ and $t$. Since $X'$ and $Y'$ have equal number of vertices, we have
\begin{equation}\label{eq:tsdivis}
m_1w+s + m_2 u - t = m_2 w + t + m_1 u - s \Leftrightarrow (m_1-m_2) (w-u)= 2(t-s). \end{equation}
Without loss of generality, assume that $m_1 \ge m_2$.  This implies $t \ge s$.

\medskip

%$h$ also divides $n$, and when $h$ does not divide $n$.

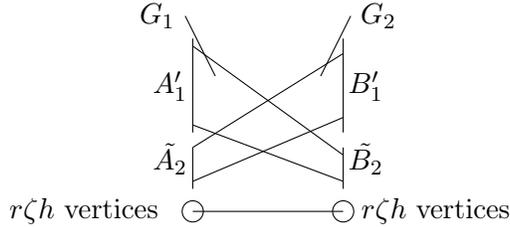
\begin{figure}[ht]
\begin{center}
\begin{tikzpicture}

\draw (-1,1) -- (-1,-.25);
\draw (-1,-.45) -- (-1,-1);
\draw (-1,-1.3) circle (4pt) node[left=9pt]{$r\zeta h$ vertices};

\draw (1,1) -- (1,-.25);
\draw (1,-.45) -- (1,-1);
\draw (1,-1.3) circle (4pt) node[right=3pt]{$r\zeta h$ vertices};

\draw (-1,-1.3) -- (1,-1.3);

\draw (-1,.9) -- (1,-.55);
\draw (-1,-.15) -- (1,-.9);

\draw (1,.8) -- (-1,-.45);
\draw (1,-.05) -- (-1,-.9);

\draw (-1.3,.37) circle (0pt) node{$A'_1$};
\draw (1.3,.37) circle (0pt) node{$B'_1$};

\draw (-1.3,-.60) circle (0pt) node{$\tilde{A_2}$};
\draw (1.3,-.60) circle (0pt) node{$\tilde{B_2}$};

\draw (-.7,.5) -- (-1.1,1.3);
\draw (.7,.5) -- (1.1,1.3);

\draw (1.1,1.3) circle (0pt) node[right]{$G_2$};
\draw (-1.1,1.3) circle (0pt) node[left]{$G_1$};

\end{tikzpicture}
\caption{Graph $G$ with sets $A'_1, \tilde{A_2}, B'_1, \tilde{B_2}$, and removed copies of $H$}
\end{center}
\end{figure}

Now we use the assumption that $m$ is even: $m-2r\zeta = m_1+m_2$ is even, thus $m_1-m_2$ is even.  Then, by \eqref{eq:tsdivis}, we see that $w-u$ divides $t-s$. We now separate the cases when $t\ge 0$ and when $t<0$.

\smallskip

\textbf{Case 1: } Assume $t \ge 0$. We claim that $t$ is reasonably small. In fact, by \cref{thm:hcfclaim}, $v(G_2)=
|A_2| + |A_0| + |B_1| - (r\zeta h -r)$. From \cref{clm:size}, we know that
$|A_2| + |B_1| \ge n - 2\alpha^{\frac{2}{3}}|A|$ and consequently $m_2=v(G_2)/h \ge (n-2\alpha^{\frac{2}{3}}|A|-r\zeta h)/{h}$. By definition,
\[
t = |B'_1| - m_2w \le |A|+\alpha^{\frac{2}{3}}|A| - \frac{w}{h}n+2\frac{w}{h}\alpha^{\frac{2}{3}}|A| + w r \zeta = \alpha^{\frac{2}{3}}|A| +2\frac{w}{h}\alpha^{\frac{2}{3}}|A| + w r \zeta
\]
By (\ref{eqn:rzh}), we have $wr\zeta\le \frac{1}{h}\a^{\frac23} |A|$ and thus $t\le 3\a^{\frac23} |A|$.

We want to move $t$ vertices from $A'_1$ to $\tilde{A_2}$ and $t$ vertices from $B'_1$ to $\tilde{B_2}$.  To move these vertices, we will find $t$ $w$-stars from $B'_1$ to $A'_1$ and $t$ $w$-stars from $A'_1$ to $B'_1$ by the following lemma from \cite{zhao1} (Lemma 12), and then move the centers of these stars.

\begin{lemma}\label[lemma]{thm:starlemma}(\cite{zhao1})
Let $1 \le k \le \delta \le M$ be positive integers, and $0 < c < \frac{1}{6k+7}$.  Let $F[V_1,V_2]$ be a bipartite graph such that $||V_i|-M| \le cM$ for $i=1,2$.  If $\delta \le \delta(V_1,V_2) \le cM$ and $\Delta(V_2,V_1) \le cM$, then $F$ contains $2(\delta-k+1)$ vertex disjoint $k$-stars of which $\delta-k+1$ are centered in $V_1$ and $\delta-k+1$ are centered in $V_2$.
\end{lemma}

By using $\d(G)\ge \frac{u}{h} n + c_1(H)$ and $m_2\le \frac{m}{2}= \frac{n}{h}$, we obtain a lower bound on $\delta(B'_1,A'_1)$:
\begin{equation}\label{eq:mindegeqn}
\delta(B'_1,A'_1) \ge \delta(G) - |\tilde{A_2}| - r\zeta h = \frac{u}{h}n + c_1(H) - m_2u + t - r\zeta h \ge c_1(H)+t - r\zeta h.
\end{equation}
By \eqref{eq:constant}, we have $c_1(H)> r\zeta h + w-1$, which implies that
\begin{equation} \label{eq:eq1} \delta(B'_1,A'_1) - w + 1 \ge c_1(H)+t- r\zeta h - w + 1 > t. \end{equation}
On the other hand, $\delta(B'_1,A'_1), \Delta(A'_1,B'_1)\le (\a^{\frac13}+ \a^{\frac23})|A|$ by Claim~\ref{clm:size}.  From \eqref{eqn:alpha} and the fact that $w \ge 2$, we can derive that $2\alpha^{\frac{1}{3}} \le \frac{2}{15(w+1)} < \frac{1}{6w+7}$.  Thus, \cref{thm:starlemma} provides $t$ vertex disjoint $w$-stars centered in $A'_1$ and $t$ vertex disjoint $w$-stars centered in $B'_1$.  We now move the centers of these stars from $A'_1$ to $\tilde{A}_2$ and from $B'_1$ to $\tilde{B}_2$.  The resulting $A'_1, \tilde{A}_2, B'_1, \tilde{B}_2$ satisfy
\[
|A'_1|=m_1w+s-t, \,\, |\tilde{B_2}| = m_1u - s + t; \quad |B'_1|=m_2w, \,\, |\tilde{A_2}|=m_2u.
\]
Below we explain how to find an $H$-factor in $G_1$; the same procedure works for $G_2$.

The resulting $G_1$ contains $t\le 3\a^{\frac23} |A|$ disjoint $w$-stars centered at $\tilde{B_2}$.  By definition, $B_0\subset \tilde{B_2}$. We next find $|B_0|$ disjoint $w$-stars centered at $B_0$ from $G_1$ which are also disjoint from the existing $w$-stars.  From Claim~\ref{clm:size}, we have $|B_0| < 2\alpha^{\frac{2}{3}} |A|$ and $\delta(B_0,A_1) \ge (\alpha^{\frac{1}{3}}-\alpha^{\frac{2}{3}})|A|$. Since $|A'_1|\ge |A_1|- r\zeta h - t$ and $r\zeta h\le \a^{\frac23} |A|$, we have
\[\d(B_0, A'_1)\ge \delta(B_0,A_1) - (t+ r\zeta h) \ge (\alpha^{\frac{1}{3}}-\alpha^{\frac{2}{3}})|A| - 3\a^{\frac23} |A| - \a^{\frac23} |A|
= (\a^{\frac13} - 5\a^{\frac23}) |A|.
\]
Since $\a^{\frac13}\ge 5h \a^{\frac23}\ge 5(w+1)\a^{\frac23}$ by (\ref{eqn:alpha}), we derive that
\[
\d(B_0, A'_1)\ge (\alpha^{\frac{1}{3}}-5\alpha^{\frac{2}{3}})|A| \ge 5w \alpha^{\frac{2}{3}} |A| \ge w(|B_0|+ t).
\]
We may therefore choose disjoint $w$-stars for the vertices of $B_0$ greedily.

Now, we have $t + |B_0|$ $w$-stars centered in $\tilde{B_2}$.  For each star, we will find a copy of $K_{u,w}$ (a supergraph of $H$), such that $u-1$ vertices come from $B'_2$, and the rest are from the $w$-star. Recall that $|B_2 - B'_2|\le r\zeta h$. Suppose that a $w$-star has leaves $v_1, \ldots , v_w$ in $A'_1$. We claim that $|\cap_{i=1}^w \Gamma(v_i, B'_2)|\ge (u-1)(|B_0|+t)$, thus we can greedily find a copy of $K_{u,w}$ for each star such that it is vertex disjoint from the existing copies of $K_{w,u}$. In fact, by \cref{clm:size} and (\ref{eqn:rzh}),
\[
|\cap_{i=1}^w \Gamma(v_i, B'_2)|\ge (1- w \tfrac{w}{u} 2\a^{\frac13})|B_2|- r\zeta h \ge
(1- \tfrac{2w^2}{u}\a^{\frac13})(1-\a^{\frac23})|B|-\a^{\frac23}|B|\ge (1- \tfrac{2w^2}{u}\a^{\frac13}- 2\a^{\frac23})|B|.
\]
By (\ref{eqn:alpha}), we have $5u\a^{\frac23}< \a^{\frac13}$ and $(\frac{2w^2}{u}+1)\a^{\frac13}< 2h^2 \a^{\frac13}< 1$. Consequently
\[
|\cap_{i=1}^w \Gamma(v_i, B'_2)|- (u-1)(|B_0|+t)\ge (1- \tfrac{2w^2}{u}\a^{\frac13}- 2\a^{\frac23})|B| - (u-1) 5\a^{\frac23} |B|> (1- \tfrac{2w^2}{u}\a^{\frac13} - \a^{\frac13})|B|>0.
\]

\medskip

We remove these copies of $K_{w,u}$, and let $A''_1$ and $B''_2$ denote the set of remaining vertices in $A'_1$ and $\tilde{B}_2$. We know that $A''_1\subseteq A_1$ and $B''_2\subseteq B_2$ satisfy
\[
|A_1|\ge |A''_1|\ge |A_1|- r\zeta h - t - w(|B_0|+t), \quad
|B_2|\ge |B''_2|\ge |B_2| - r\zeta h - (u-1)(|B_0|+t).
\]
Furthermore, $|A''_1|= m'_1 w + s-t$ and $|B''_2|= m'_1 u - s+t$ for some large integer $m'_1$. Since by \cref{thm:newlemma9} (which we can apply since $w-u$ divides $t-s$), $K_{|A''_1|, |B''_2|}$ contains an $H$-factor, if
$G[A''_1, B''_2]$ satisfy the condition (\ref{eqn:1-r}) of  \cref{thm:modblowuplemma}, then \cref{thm:modblowuplemma} provides an $H$-factor of $G[A''_1, B''_2]$.

In fact, by \cref{clm:size},
\begin{align*}
\d(B''_2, A''_1) &\ge (1- 2\a^{\frac13})|A_1| - r\zeta h - t - w(|B_0|+t)\\
&\ge (1- 2\a^{\frac13})|A_1| - \a^{\frac23} |A| - 3\a^{\frac23} |A| - w (5\a^{\frac23} |A|).
\end{align*}
By (\ref{eqn:alpha}) and $w+1\le h$, we have $\a^{\frac13}\ge 5 (w+1)\a^{\frac23}$, which implies that, by \cref{clm:size},
\[\a^{\frac13} |A_1|\ge \a^{\frac13} (1-\a^{\frac23}) |A|\ge (4\a^{\frac23}+ 5w\a^{\frac23}) |A|.
\]
Consequently $\d(B''_2, A''_1)\ge (1- 3\a^{\frac13})|A_1|\ge (1-3\a^{\frac13})|A''_1|$.

On the other hand,
\begin{align*}
\d(A''_1, B''_2) &\ge (1- \frac{2w}{u}\a^{\frac13})|B_2| - r\zeta h - (u-1)(|B_0|+t)\\
&\ge (1- \frac{2w}{u}\a^{\frac13})|B_2| - \a^{\frac23} |B| - (u-1) 5\a^{\frac23} |B|.
\end{align*}
By (\ref{eqn:alpha}), we have $\a^{\frac13} \frac{u}{w}\ge 5 u\a^{\frac23}$. Together with $|B_2|\ge |B^c| - \a^{\frac23} |B|\ge \left(\frac{u}{w}-\a^{\frac23}\right) |B|$, we have
\[\a^{\frac13} |B_2|\ge \a^{\frac13} \left(\frac{u}{w}-\a^{\frac23}\right) |B|\ge \a^{\frac23} |B| + 5(u-1)\a^{\frac23} |B|.
\]
Consequently $\d(A''_1, B''_2)\ge (1- \frac{2w}{u}\a^{\frac13}- \a^{\frac13})|B_2|\ge (1-2h\a^{\frac13})|B''_2|$. By (\ref{eqn:alpha}), we have $3\a^{\frac13}\le 2h\a^{\frac13}\le \rho$, and thus
$\d(B''_2, A''_1)\ge (1- \rho)|A''_1|$, and $\d(A''_1, B''_2)\ge (1 -\rho)|B''_2|$, as stated in (\ref{eqn:1-r}).

%\left( \dfrac{u}{h} - 2\alpha^{\frac{1}{3}}\dfrac{w}{h} \right)n (u-1)|A_0| \]
%
% By applying \cref{thm:modblowuplemma}, we get an $H$-factor of $G_1$ and $G_2$.

\medskip

\textbf{Case 2: } Assume $t < 0$.  Let $-t = q(w-u)+p$ for some nonnegative integers $q$ and $p$ such that $p < w-u$. Since $-s\ge -t$ and $w-u$ divides $t-s$, we may write $-s=q'(w-u)+p$ for some integer $q'\ge q$. Similar as in Case~1, we derive that $-s\le 3\a^{\frac23}|A|$.

First, assume that $q \ge p\beta$. Then by \cref{thm:newlemma9}, $K_{|A'_1|, |\tilde{B}_2|}$ and $K_{|B'_1|, |\tilde{A}_2|}$ each contains an $H$-factor (here we need $n\gg -s, -t$). In order to obtain an $H$-factor in $G_1 = G[A'_1, \tilde{B}_2]$ (similar for $G_2= G[B'_1, \tilde{A}_2]$), as in Case~1, we first find $|B_0|$ disjoint $w$-stars with centers at $B_0$ and leaves in $A'_1$. Then we extend these $w$-stars to (disjoint) copies of $K_{w, u}$ and finally apply \cref{thm:modblowuplemma} to find an $H$-factor covering the remaining part of $G_1$.

Secondly, assume that $q \le p\beta-1$. We will move $w-u-p$ vertices from $A'_1$ to $\tilde{A_2}$, and $w-u-p$ vertices from $B'_1$ to $\tilde{B_2}$. As a result,
\[
|A'_1|= m_1 w + s - (w-u-p) = m_1 w - (q'+1)(w-u), \hfill
|\tilde{A}_2|= m_2 u - t + (w-u-p) = m_2 u + (q+1)(w-u),
\]
\[|B'_1|= m_2 w - (q+1)(w-u), \quad |\tilde{B}_2|= m_1 u + (q'+1)(w-u).
\]
By \cref{thm:newlemma9}, $K_{|A'_1|, |\tilde{B}_2|}$ and $K_{|B'_1|, |\tilde{A}_2|}$ both contains an $H$-factor. Then we can find an $H$-factor of $G_1$ and $G_2$ as above.
We now explain how to find such $w-u-p$ vertices from $A'_1$ and from $B'_1$. Similar as in Case~1, we use \cref{thm:starlemma} to find $2(w-u-p)$ vertex-disjoint $w$-stars in $G[A'_1, B'_1]$ with $w-u-p$ of them centered at $A'_1$ and the other $w-u-p$ stars centered at $B'_1$. It remains to show that $\delta(B_1,A_1)- w+1 \ge w-u-p$.  By \eqref{eq:constant}, we have $c_1(H) > p\beta(w-u) + r\zeta h + w\ge (q+1)(w-u) + r\zeta h + w$.  With \eqref{eq:mindegeqn}, this implies that
\begin{equation}\label{eq:eq2}
\delta(B_1,A_1) - w + 1 > c_1(H)+t-r\zeta h - w + 1 \ge w-u-p
\end{equation}

\bigskip

\textbf{Part II: } Assume $m$ is odd.  In this case we use an idea used in the proof of Lemma 16 in \cite{kuhnosthus}:  we will use $hcf_c(H)=1$ to remove a small number of copies of $H$ such that the remaining vertices of $G$ form a balanced, bipartite graph of size $2n'=m'h$ where $n'$ is divisible by $H$.  Then, we apply the proof of \textbf{Part I} to this graph, and complete our tiling.

Because $m$ is odd and $mh=2n$ is even, then $h$ must be even.  Moreover, since $hcf_c(H)=1$, there exists a component $C_i[U_i,W_i]$ of $H$ with an odd number of vertices.  Since $c_i$ is odd, $w_i - u_i$ is odd.  Now, take the $2$-coloring $c_1$ of $H$ with color classes $U$ and $W$ (then $|U|=u$, $|W|=w$) such that $U_i \subset U$, $W_i \subset W$.  We obtain another coloring $c_2$ of $H$ by swapping the colors of $U_i$ and $W_i$ from $c_1$. Suppose that $c_2$ has color classes $U'$ and $W'$ such that $|U'|=u'\le w'= |W'|$.  Since $h$ is even, $u$ and $w$ have the same parity, and $u'$ and $w'$ have the same parity.  Additionally, since $w_i - u_i$ is odd, the parities of $u, w$ and $u', w'$ are different.  %suppose $u$ and $w$ are odd, and $u'$ and $w'$ is even.

Let $k_1=\frac{h}{2}-u'$ and $k_2=\frac{h}{2} - u$ (so $k_1, k_2\ge 0$). From $G[A_1,B_2]$, remove $k_1$ copies of $H$ with $u$ vertices in $A_1$ and $w$ vertices in $B_2$, and remove $k_2$ copies of $H$ with $w'$ vertices in $A_1$ and $u'$ vertices in $B_2$. This is possible because $G[A_1,B_2]$ is almost complete. Denote the sets of the remaining vertices in $X$ and $Y$ by $X'$ and $Y'$, respectively.

We first observe that $|X'|= |Y'|$. Since $|X|=|Y|$, it suffices to show that $|X|-|X'|= |Y|- |Y'|$. In fact, since $|X|-|X'|= k_1 u + k_2 w'$ and $|Y|- |Y'|= k_1w+k_2 u'$, by
the definitions of $k_1$ and $k_2$,
\begin{align*}
& k_1 u + k_2 w' = k_1 w + k_2 u'\\
& \Leftrightarrow \left( \frac{u'+w'}{2} - u' \right ) u + \left ( \frac{u+w}{2}-u \right ) w' = \left ( \frac{u'+w'}{2} - u' \right ) w + \left ( \frac{u+w}{2} - u \right ) u' \\
& \Leftrightarrow \frac{w'-u'}{2}u + \frac{w-u}{2}w' = \frac{w'-u'}{2}w + \frac{w-u}{2} u',
\end{align*}
which is equivalent to the identity $\frac{1}{2}(w'-u')(w-u)=\frac{1}{2}(w'-u')(w-u)$.

Let $n'= |X'|=|Y'|$. We have $n-n' = (k_1 + k_2)h/2= (h- u - u')h/2$. Since $h$ is even and $u+u'$ is odd, we have $n-n'\equiv \frac{h}{2}$ mod $h$. Since $n=mh/2\equiv \frac{h}{2}$ mod $h$, we derive that $n'$ is divisible by $h$. Furthermore, since $u'\ge u$, we have $n= n - (h- u - u')h/2\ge n- (w-u)h/2$ .

In the new graph $G'= G[X', Y']$, we have $\delta(G') \ge \frac{u}{h} n + c_1(H) - (w-u)h/2$. By \eqref{eq:constant}, $c_1(H) \ge \beta(w-u)^2 + \zeta h^2 + (w - u)\frac{h}{2} + w$, and so we have $\delta(G') \ge \frac{u}{h}n + c_1(H)$, where $c_1(H)\ge \beta(w-u)^2 + \zeta h^2 + w$. Hence \eqref{eq:eq1} and \eqref{eq:eq2} hold and we may apply the proof of Part I to $G'$  obtaining an $H$-factor.

%\begin{align*} |A|-|A'| = k_1 u + k_2 w' = (\frac{1}{2}(u+w) - u')(u)+(\frac{1}{2}(u'+w')(w')  |B|-|B'|=k_1w+k_2 u'  \\
% \end{align*}
%
%Substituting in $\frac{h}{2}=\frac{1}{2}(u+w)$ into the $k_1$ terms, and $\frac{h}{2}=\frac{1}{2}(u'+ w')$ into the $k_2$ terms, we get $|A'|=|B'|$.

\end{proof}

\section{Proof of \cref{thm:almost}}
Let $H$ be a bipartite graph on $h$ vertices with $u=\sigma(H)$ and $w=h-\sigma(H)$.
Let $G$ be a balanced bipartite graph on $2n$ vertices with $\delta(G) \ge \frac{u}{h} n$. We assume $u<w$ otherwise we can obtain the desired $H$-tiling as follows. Add $3h$ new vertices to each side of $G$ and join them with all the existing vertices on the opposite side. The new graph $G'$ has $\d(G')\ge \frac{n}2 + 3h = \frac{n+3h}2 + \frac{3h}2$. By \cref{thm:zhao}, $G'$ contains an $H$-factor $\cal H$, which gives rise to an $H$-tiling of $G$ that misses at most $6h(h-1)$ vertices because at most $6h$ copies of $H$ in $\cal H$ may contain the vertices of $G'-G$, and each copy of $H$ may contain at most $h-1$ vertices of $G$.

Part 1 of the following lemma is a replacement of \cref{thm:kuhnosthuslemma} when $hcf_{\chi, c}(H)\neq 1$; Part 2 is needed for the extremal case.

\begin{lemma}\label[lemma]{lem:alm}
\begin{enumerate}
    \item Let $G[X,Y]$ be a complete bipartite graph with $\frac{u}{w} \le \frac{|X|}{|Y|} \le 1$. Then $G$ has a $K_{u,w}$-tiling that leaves out $l(X)$ vertices in $X$ and $l(Y)$ vertices in $Y$ such that $l(X) + l(Y)\le h+(w-u)-2$. In this $K_{u,w}$-tiling, at least $m/2 -h$ copies of $K_{u,w}$ have their $w$-vertex sides in $Y$, where $m=\lfloor \frac{|X|+|Y|}{h} \rfloor$.
    \item Let $m > c$ be positive integers. Then $G[X, Y] = K_{mu-c, mw+c}$ contains a $K_{u,w}$-tiling that covers all but at most $(c+u-1)\frac{h}{u}$ vertices.
\end{enumerate}
\end{lemma}

\begin{proof}
\textbf{Part 1:} Let $r \equiv |X|+|Y| \mod{h}$ (then $0 \le r \le h-1$).  We may write $|X|=mu+t$ and $|Y|=mw-t+r$.  Since $\frac{|X|}{|Y|} \ge \frac{u}{w}$, we have
$|X| \ge (|X|+|Y|)\frac{u}{h} \ge mu$, which implies that $t \ge 0$.  We next write $t=q(w-u) + p$ for some integers $q$ and $0 \le p \le w-u-1$.  We now have two cases.

First, if $p \le r$, then we may tile $G$ with $m$ copies of $K_{u,w}$ where $m-q$ copies have their $w$-vertex sides placed in $Y$, and $q$ copies have their $w$-vertex sides placed in $X$.  This tiling covers $(m-q)w+q u = m w-t+p = |Y|-(r-p)$ vertices of $Y$ and $(m-q)u+qw=mu+t-p=|X|-p$ vertices of $X$. Let $l(X)=p$ and $l(Y)= r-p$. We have $l(X) + l(Y) =r\le h-1$.

Otherwise, $p > r$.  In that case, tile $G$ with $m-q-1$ copies of $K_{u,w}$ with their $w$-vertex sides placed in $Y$, and $q$ copies of $K_{u,w}$ with their $w$-vertex sides placed in $X$. This tiling covers $(m-q-1)w+ q u = m w -(t-p)-w=|Y|+p-(r+w)$ vertices of $Y$ and $(m-q-1)u+ qw = mu+t-(p+u)=|X|-(p+u)$ vertices of $X$. Let $l(X)= p+u$ and $l(Y)= r+w-p$. We have $l(X) + l(Y)= r+h\le h+w-u-2$ since $r < p\le w-u-1$.

In both cases, our $H$-tiling contains at least $m-q-1$ copies of $K_{u,w}$ with their $w$-vertex sides in $Y$. Since $|X|\le |Y|$, we have $mu + t\le mw - t+r$, or $2t\le m(w-u)+r$. With $t=q(w-u) + p$, this gives $m\ge 2q + \frac{2p-r}{w-u}$. By using $r\le h-1$, we have $m- q- 1\ge \frac{m}2 - \frac{h-1}{2(w-u)} -1\ge  \frac{m}2 - h$.

\textbf{Part 2:} Write $c= pu + q$ for integers $p, q$ such that $0\le q<u$. If $q=0$, then $|X|=mu-pu$ and $G \supset K_{(m-p)u, (m-p)w}$, which consists of $m-p$ copies of $K_{u,w}$. It leaves $c+ p w= p h = ch/u$ vertices in $Y$ uncovered. Otherwise $q\ge 1$ and $G\supset K_{(m-p-1)w, (m-p-1)u}$, which consists of $m-p-1$ copies of $K_{u,w}$. It leaves $u-q$ vertices in $X$ and $c+ (p+1)w$ vertices in $Y$ uncovered. The total number of uncovered vertices is
\[
u - q + c+ (p+1)w = u + p u + (p+1)w = h + p h = h \left(\frac{c-q}{u} + 1 \right)\le (c+u-1)\frac{h}{u}. \qedhere
\]
\end{proof}

\medskip

\begin{proof}[Proof of \cref{thm:almost}]
First note what is different here from \cref{thm:main}: (1). we do not assume that $hcf(H)=1$; (2) the $\d(G)$ condition has no extra constant $c_1(H)$; (3) at most $c_2(H)$ vertices may be left outside the desired $H$-tiling. Below we closely follow the proof of \cref{thm:main} but focus on the impact of these differences.

\smallskip

\textbf{Non-extremal Case:} We assume $G$ is not in the extremal case, which is defined exactly as in \cref{thm:main}. First note that \cref{thm:nonextremalcase} has no $c_1(H)$ in the minimum degree condition, and \cref{thm:broadlemma1} does not assume that $hcf(H)=1$.   We thus apply \cref{thm:broadlemma1} to get a decomposition of $G$ into super-regular cluster pairs $(P_1, Q_1), \dots, (P_k, Q_k)$, and exceptional sets $X_0, Y_0$.  We can not apply \cref{thm:broadlemma2} directly because it assumes that $hcf(H)=1$. If we follow the proof of \cref{thm:broadlemma2}, we can apply  Claim~\ref{clm:X0Y0} to get rid of the exceptional sets but we can not use Claim~\ref{clm:div} because we do not have $hcf_c(H)=1$. Actually even if $h$ divides $|P_i|+|Q_i|$, we can not use \cref{thm:kuhnosthuslemma} to obtain an $H$-factor on $P_i\cup Q_i$ because we do not have $hcf_{\chi, c}(H)=1$. Instead we can only apply \cref{lem:alm} to obtain an $H$-tiling that omits at most $h+(w-u)-2$ vertices of $P_i\cup Q_i$. If we apply \cref{lem:alm} to each $(P_i, Q_i)$, then we obtain an $H$-tiling of $G$ that omits at most $2h k$ vertices, where $k\le 2p M(\epsilon)$ is a large constant depending on the large constant $M(\epsilon)$ defined in the Regularity Lemma.

In order to reduce the number of uncovered vertices to a constant $O(h^2)$, we use the connection among $P_i, Q_i, i=1, \dots, k$ to gather all uncovered vertices in a few cluster pairs. This approach can be found in \cite{zhao2}. To facilitate our calculation, we need all $P_i$ (and thus all $Q_i$) to have the same size. Let us go back to the moment right after we decompose the clusters of $R$. As shown in \eqref{eq:sizes}, there are only a few possible sizes for $P_i$, and $\frac{N_0}{q(q^2- p^2)}$ divides all of them. We then divide each $P_i$ to subclusters of size $N_1:= \frac{N_0}{q(q^2- p^2)}$,  accordingly divide its partner $Q_i$ to subclusters of size $N_2 := \frac{q}{p} N_1$, and match the resulting subclusters from $P_i$ and those from $Q_i$ arbitrarily. Let us still denote new cluster pairs by $(P_i, Q_i)$, and use $k$ for the number of the new cluster pairs. Let $k_1$ be the number of $(P_i, Q_i)$ with $P_i\subset X$. We have $k_1= k/2$ because there are the same number of vertices of $G$ contained in the clusters of $X$ and in the clusters of $Y$ (note that$|X_0|=|Y_0|$).
We call $P_i$ and $Q_i$ the \emph{partners} of each other. To distinguish them, we call $P_1, \dots, P_k$ \emph{small} clusters and $Q_1, \dots, Q_k$ \emph{large} clusters.

Now let $R'$ be the bipartite graph on $\{P_i, Q_i : i=1, \dots, k\}$ such that two clusters $C, C'$ are adjacent if $d(C, C')>0$ where we consider the density after applying the Regularity Lemma. Consider a vertex $C\in V(R')$. Since each cluster, $P_i$ or $Q_i$, has at most $N_2$ vertices, by the same calculation as in \eqref{eq:dR}, we derive that $\d_{R'}(C)\ge (u/h - 2\g) n/N_2$. Since $N_1 \frac{k}2 + N_2 \frac{k}2 =\sum_{C\subset X} |C| \le n$ and $N_1 > N_2 \frac{u}{w}$, we obtain that $\frac{n}{N_2} > (1+ \frac{u}{w}) \frac{k}2 = \frac{h k}{2w}$. Consequently $\d_{R'}(C)\ge (\frac{u}{2w} - \frac{h}{w} \g)k$.

We next define a directed graph $D_X$ whose vertices are all the current clusters in $X$, namely, $P_1, \dots, P_{k/2}$, $Q_{k/2+1}, \dots, Q_k$, and direct an edge from a cluster $C$ to another $C'$ if and only if $d(C, C'')>0$, where $C''$ is the cluster in $Y$ matched to $C'$. Then the minimum out-degree $\d(D_X)= \d_{R'}(C)\ge(\frac{u}{2w} - \frac{h}{w} \g)k$. Define the \emph{sink} of $D_X$ as a subset $S\subseteq V(D_X)$ such that for every vertex $v\in V(D_X)$, there is a vertex $s\in S$ and a directed path from $v$ to $s$. A simple fact on digraphs (e.g. Lemma 6.7 in \cite{zhao2}) states that every digraph $D$ contains a sink of size at most $|D|/\d(D)$. Then $D_X$ has a sink $S_X$ of size at most
%\left\lfloor \frac{k}{\d(D_X)}\right\rfloor\le \left\lfloor \frac{k}{(\frac{u}{2w} - \frac{h}{w} \g)k}\right\rfloor = \frac{2w}{u}
\[
\frac{k}{\d(D_X)} \le\frac{k}{(\frac{u}{2w} - \frac{h}{w} \g)k} = \frac{2w}{u- 2h\g}.
\]
Since $\g\ll 1$, this implies $|S_X|\le 2w/u$. We similarly define the digraph $D_Y$ on all the clusters of $Y$ and obtain a sink $S_Y$ of size at most $2w/u$. Let $M_S$ be the set of all cluster pairs that contain at least one member of $S_X\cup S_Y$. Then $|M_S|\le 4w/u$.

After this detour, we go back to the proofs of Lemmas~\ref{thm:broadlemma1} and \ref{thm:broadlemma2}: we obtain the super-regularity of all $(P_i, Q_i)$ as in the proof \cref{thm:broadlemma1} and then eliminate the exceptional sets $X_0, Y_0$ by Claim~\ref{clm:X0Y0}. Note that in these steps we only remove a small number of vertices from each cluster and thus do not change the adjacency in $R', D_X, D_Y$. Now all  $(P_i, Q_i)$ are super-regular and ratios $|P_i|/|Q_i|$ are slightly larger than $u/w$.
%Because of the Blow-up Lemma, we may treat each $(P_i, Q_i)$ as a complete bipartite graph when tiling $(P_i, Q_i)$ with copies of $H$.
Let $l(P_i)$ and $l(Q_i)$ be the numbers of leftover vertices in $P_i$ and $Q_i$ when we apply \cref{lem:alm} to $K_{|P_i|, |Q_i|}$ %\footnote{In fact $l$ is a function of $|P_i|$ and $|Q_i|$.}
(then $l(P_i)+ l(Q_i)\le h+ w-u-2$). Since $|P_i|+|Q_i|$ is sufficiently large, by \cref{lem:alm}, the values of $l(P_i), l(Q_i)$ do not change after we remove $c u$ vertices from $P_i$ and $c w$ vertices from $Q_i$ for any fixed integer $c$.

Before actually tiling $(P_i, Q_i)$, we remove $l(C)$ vertices from each $C$ not included in $M_S$ as follows. Assume that $C\subset X$ and $l(C)= l_0$. By the definition of $S_X$, there is a directed path $C_0 C_1 \dots C_t$ from $C_0:= C$ to some $C_t\in S_X$ in $D_X$. Let $C'_j$ denote the partner of $C_j$ for $1\le j\le t$. For $0\le j< t$, we find $l_0$ disjoint copies of $K_{u, w}$, each of which consists one vertex of $C_j$ and $w+u-1$ vertices from $C_{j+1}\cup C'_{j+1}$ such that $C_{j+1}$ loses $u-1$ vertices if it is small or loses $w-1$ vertices if it is large. At the end, $C_0$ loses $l_0$ vertices, $C_t$ loses $l_0 (u-1)$ vertices (if it is small) or $l_0 (w-1)$ (if it is large) while any of the clusters $C_1, \dots, C_{t-1}, C'_1, \dots, C'_t$ loses $l_0 u$ vertices (if it is small) or $l_0 w$ vertices (if it is large).  As a result, $l(C)$ becomes zero %$l(C_t)$ increases by $l_0$
while $l(C_1), l(C'_1), \dots, l(C_{t-1}), l(C'_{t-1})$ stay the same.
We apply this procedure to every cluster $C$ not included in $M_S$ such that $l(C)=0$ at the end.  Note that each cluster loses constant many (at most $4k h w$) vertices even if it is contained in all the directed paths because there are at most $2k$ paths, and each path uses at most $(2h) w$ vertices from a single cluster. Hence the resulting cluster pairs are still super-regular and satisfy $u/w\le |P_i|/|Q_i|\le 1$. Now we apply \cref{lem:alm} and the Blow-up Lemma to each $(P_i, Q_i)$ and obtain an perfect $H$-tiling unless $(P_i, Q_i)\in M_S$. Since each cluster pair in $M_S$ contains at most $h+w-u-2$ uncovered vertices, we obtain an $H$-tiling of $G$ that misses at most $|M_S|(h+w-u-2)\le \frac{4w}{u} (h+w-u-2)< 8h^2$ vertices.

\medskip

\textbf{Extremal Case:} Following the proof of \cref{thm:extremalcase}, we first define $A_i, B_i$ for $i=0,1,2$. \cref{clm:size} still holds because it only needs $\d(G)\ge \frac{u}{h} n$. Then we do not need to separate the cases on the parity of $m$. Define $G_1= G[A_1, B_2\cup B_0]$ and $G_2 = G[B_1, A_2\cup A_0]$ as well. Assume that $v(G_1)\equiv r \mod{h}$ and $v(G_2)\equiv h-r \mod{h}$ for some $0\le r< h$.
We remove arbitrary $h$ vertices from $A_1$, $h-r$ vertices from $B_1$ and $r$ vertices from $B_2$ and ignore them permanently.
%put them into two {\em trash} sets $X_0$ and $Y_0$, which will contain all uncovered vertices in $X$ and $Y$.
Denote the sets of the remaining vertices by $X', Y', A'_1, A'_2, B'_1, B'_2$. Then $|X'|=|Y'|= n-h$. Let $\tilde{A_2}= A'_2 \cup A_0$ and $\tilde{B_2}= B'_2 \cup B_0$. Update $G_1, G_2$ as $G[A'_1, \tilde{B_2}]$ and $G[B'_1, \tilde{A_2}]$, respectively. Since both $v(G_1)$ and $v(G_2)$ are divisible by $h$, we have
\begin{equation}\label{eq:17}
|A'_1|=m_1w + s, \quad |B'_1|=m_2w+t, \quad |\tilde{A_2}|=m_2u-t, \quad |\tilde{B_2}|=m_1u-s
\end{equation}
for some integers $m_1, m_2, s, t$. Without loss of generality, assume that $m_1 \ge m_2$ and consequently $t \ge s$. Let $c_0= h+w-1$. We separate the cases when
$t\le c_0$ and $t> c_0$.

First assume that $t\le c_0$ (so $s\le t\le c_0$). As in the proof of \cref{thm:extremalcase}, we remove $|A_0|+|B_0|$ copies of $K_{w,u}$ from $G_1$ and $G_2$, each of which contains a vertex from $A_0\cup B_0$, such that \eqref{eq:17} holds for (slightly) smaller values of $m_1$ and $m_2$. If
$t<0$, then we have
\[
    \frac{u}{w}\le \frac{|\tilde{B}_2|}{|A_1'|}, \frac{|\tilde{A}_2|}{|B_1'|} < 1.
\]
By Lemma~\ref{lem:alm}, Part 1, $K_{|\tilde{B}_2|, |A_1'|}$ and $K_{|\tilde{A}_2|, |B_1'|}$ each contains an $H$-tiling that misses at most $h+(w-u)-2$ vertices. Consequently, by Lemma~\ref{thm:modblowuplemma}, $G_1$ and $G_2$ contain the same $H$-tilings. The number of uncovered vertices in this case is at most $2h + 2(h+w-u-2)$. If $0\le t\le c_0$, then by Lemma~\ref{lem:alm}, Part 2, and Lemma~\ref{thm:modblowuplemma}, each of $G_1$ and $G_2$ contains an $H$-tiling that misses at most $(c_0+ u-1)h/u= (h+w-1+u-1)h/u\le (2h-2)h$ vertices. The
total number of uncovered vertices in this case is at most $2h + 2(2h-2)h$.

Now assume that $t> c_0$. After removing $c_1(H)$ and replacing $r\zeta h$ by $h$
%(recall that $|X'|=|Y'|= n - h$)
in \eqref{eq:mindegeqn}, we obtain that $\d(B'_1, A'_1)\ge t-h$.
Applying Lemma~\ref{thm:starlemma}, we find $2(t-h-w+1)= 2(t-c_0)$ vertex disjoint $w$-stars with $t-c_0$ of them centered at $A'_1$ and other $t-c_0$ of them centered at $B'_1$. After moving the centers of these stars to $\tilde{A}_2$ and $\tilde{B}_2$, we have
\[
|A'_1|=m_1w+ s-t+c_0, \,\, |\tilde{B_2}| = m_1 u - s + t- c_0; \quad |B'_1|=m_2w+ c_0, \,\, |\tilde{A_2}|=m_2u - c_0.
\]
After getting rid of $A_0\cup B_0$ as before, we apply Lemma~\ref{thm:modblowuplemma} together with Lemma~\ref{lem:alm}, Part 2, to obtain $H$-tilings in $G_1$ and $G_2$, each of which misses at most $(c_0+ u-1)h/u \le (2h-2)h$ vertices (note that $s-t+c_0\le c_0$). The total number of uncovered vertices in this case is at most $2h + 2(2h-2)h$.

In summary, the number of uncovered vertices the extremal case is at most $2h + 2(2h-2)h< 4h^2$.
\end{proof}

\section{Concluding Remarks}
In summary, we determine the minimum degree threshold for bipartite tiling as follows. Given a bipartite graph $H$ of order $h\ge 2$, let $\delta_2(n, H)$
denote the smallest integer $k$ such that every balanced bipartite graph $G$ of order $2n$, which is divisible by $h$, with $\delta(G)\ge k$ contains an $H$-factor.
%Assume that $n$ is sufficiently large.
Theorems~\ref{thm:lb} and \ref{thm:main} together imply that %as $n\to \infty$,
\[
\delta_2(n,H) = \left\{
\begin{array}{rl}
( 1- 1/{\chi_{cr}(H)}) n + O(1) & \text{if } hcf(H)=1\\
( 1- 1/{\chi(H)}) n + O(1) & \text{otherwise},
\end{array} \right.
\]

As explained before, Theorem~\ref{thm:main} implies an approximate version of Theorem~\ref{thm:ko} for bipartite $H$, in which the constant $C$ is replaced by $o(n)$. In fact, if the following conjecture of Bollob\'as and Scott \cite{bolscottjudpart1} is true, we can even get Theorem~\ref{thm:ko} exactly.

\begin{conjecture}[\cite{bolscottjudpart1}]
If $G$ is a graph of even order, then $G$ contains a spanning, balanced bipartite subgraph $B$ such that for every vertex $v$ in $G$, $d_B(v)\ge \frac{d_G(v)}{2}-\frac{1}{2}$.
\label[conjecture]{conj:BS}
\end{conjecture}

In fact, for this purpose, it suffices to have a weaker form of \cref{conj:BS}:  every graph $G$ contains a spanning, balanced, bipartite subgraph $B$ such that $\delta(B) \ge \frac{\delta(G)}{2}-c$, where $c$ is some absolute constant.  %This in conjunction with Theorem~\ref{thm:main} implies Theorem~\ref{thm:ko} exactly for bipartite graphs.

After seeing the similarity between Theorem~\ref{thm:ko} and Theorems~\ref{thm:main}, it is reasonable to expect such a result for $r$-partite tiling. In an $r$-partite graph $G$, we define the \emph{pairwise minimum degree} $\bar{\delta}(G)$ as the minimum degree from a vertex in one partition set to any other partition set.
\begin{conjecture}
\label{conj:r}
Let $H$ be a graph with order $h$ and chromatic number $r$. There exist integers $C$ and $m_0$ such that for all $m\ge m_0$, if $G$ is a balanced $r$-partite graph with $n=mh$ vertices in each partition set such that
\[ \bar{\delta}(G) \ge \left\{
\begin{array}{rl}
( 1- 1/{\chi_{cr}(H)}) n + C & \text{if } hcf(H)=1\\
( 1- 1/{\chi(H)}) n + C & \text{otherwise},
\end{array} \right.
\]
then $G$ contains an $H$-factor.
\end{conjecture}

At present Conjecture~\ref{conj:r} is out of reach as it has not been confirmed for $H= K_r$ with $r>4$. In other words, we do not have the multipartite version of the Hajnal-Szemer\'edi theorem. This problem was studied by Fischer \cite{Fischer}, who obtained an almost perfect tiling for the case of $K_3$ and $K_4$. Magyar and Martin \cite{MM} proved Conjecture~\ref{conj:r} for $K_3$ with $C=1$; Martin and Szemer\'edi \cite{MSz} proved Conjecture~\ref{conj:r} for $K_4$ with $C=0$. Csaba and Mydlarz \cite{Csaba} recently proved an approximate version
of Conjecture~\ref{conj:r} for $H=K_r$ in which they assume $\bar{\d}(G)\ge \frac{k_r}{k_r+1} n$, where $k_r= r + O(\log r)$. Furthermore, Martin and Zhao \cite{MZ} proved Conjecture~\ref{conj:r} for all complete tripartite graphs $K_{s,s,s}$. Given the success on the tiling of $K_3$ and $K_4$, it may not be very hard to prove Conjecture~\ref{conj:r} for all $3$-chromatic or $4$-chromatic $H$.

%\pagebreak
%\renewcommand{\bibname}{\normalsize \centerline{wIBLIOGRAPHY}}
%\addcontentsline{toc}{chapter}{wIBLIOGRAPHY}

\end{document}